\documentclass[twoside,10pt]{amsart}

\usepackage{amsmath}
\usepackage{xspace}
\usepackage{amssymb}
\usepackage{bbm}
\usepackage{graphicx}
\usepackage{url}
\usepackage{pifont}
\usepackage{enumerate}
\usepackage{tikz}
\usepackage{lipsum}
\usepackage[matrix,arrow]{xy}

\usepackage{bbm}
\usepackage{verbatim}
\usepackage{graphicx,color}
\usepackage{enumerate}
\usepackage[hyperindex=true,plainpages=false,colorlinks=false,pdfpagelabels]{hyperref}

\theoremstyle{plain}
\newtheorem{The}{Theorem}
\newtheorem*{The*}{Theorem}
\newtheorem{Pro}{Proposition}
\newtheorem{Lem}{Lemma}
\newtheorem{Cor}{Corollary}
\newtheorem*{Cor*}{Corollary}

\theoremstyle{definition}
\newtheorem*{Def}{Definition}
\newtheorem{Rem}{Remark}
\newtheorem{Exa}{Example}
\newtheorem*{Rem*}{Remark}
\newtheorem*{Con*}{Convention}

\numberwithin{equation}{section}

\def\Res{{\,\rm Res}}
\DeclareMathOperator{\tr}{tr}

\DeclareMathOperator{\Hol}{Hol}

\DeclareMathOperator{\Id}{Id}

\renewcommand{\sl}{\mathfrak{sl}}

\newcommand{\pd}[2]{\frac{\partial #1}{\partial #2}}

\renewcommand{\Im}{\operatorname{Im}}

\newcommand{\R}{\mathbb{R}}

\newcommand{\C}{\mathbb{C}}
\newcommand{\N}{\mathbb{N}}
\newcommand{\Z}{\mathbb{Z}}

\renewcommand{\S}{\mathbb{S}}

\newcommand{\CP}{\mathbb{CP}}

\newcommand{\wt}{\widetilde}
\newcommand{\wh}{\widehat}
\newcommand{\WWW}{\mu(\varphi)}
\newcommand{\TTT}{\nu(\varphi)}

\DeclareFontFamily{U}{mathx}{\hyphenchar\font45}
\DeclareFontShape{U}{mathx}{m}{n}{
      <5> <6> <7> <8> <9> <10>
      <10.95> <12> <14.4> <17.28> <20.74> <24.88>
      mathx10
      }{}
\DeclareSymbolFont{mathx}{U}{mathx}{m}{n}
\DeclareFontSubstitution{U}{mathx}{m}{n}
\DeclareMathAccent{\widecheck}{0}{mathx}{"71}
\DeclareMathAccent{\widetilde}{0}{mathx}{"72}
\DeclareMathAccent{\widebar}{0}{mathx}{"73}
\DeclareMathAccent{\widevec}{0}{mathx}{"74}
\DeclareMathAccent{\widehat}{0}{mathx}{"70}
\DeclareMathAccent{\widefrown}{0}{mathx}{"75}
\DeclareMathAccent{\chinesehat}{0}{mathx}{"69}
\newcommand{\wc}[1]{\widecheck{#1}}
\newcommand{\wf}[1]{\widefrown{#1}}
\newcommand{\PhiDPW}{\mathsf \Phi}

\begin{document}

\title[Chern-Simons gauge theory and the enclosed volume of CMC surfaces in the 3-sphere]{Application of Chern-Simons gauge theory to the enclosed volume of constant mean curvature surfaces in the 3-sphere}

\author{Lynn Heller}

 \author{Sebastian Heller}

\author{Martin Traizet}

\address{Lynn Heller \\BIMSA\\China
 }
 \email{lynn@bimsa.cn}

\address{Sebastian Heller \\BIMSA\\China
 }
 \email{sheller@bimsa.cn}

\address{Martin Traizet \\Institut Denis Poisson, Universit\'e de Tours\\France
 }
 \email{martin.traizet@univ-tours.fr}




\thanks{}

\begin{abstract}
Building on Hitchin's work of the Wess-Zumino-Witten term for harmonic maps into Lie groups,
we derive a formula for the enclosed volume of a compact CMC  
surface $f$ in $\mathbb S^3$ in terms of a holonomy on the Chern-Simons bundle and the Willmore functional. By construction the  enclosed volume only depends on the gauge classes of the associated family of flat connections of $f$. In this paper we show in various examples the effectiveness of this formula, in particular for surfaces of genus $g\geq2.$
\end{abstract}
  
 \maketitle


\section{Introduction}

The study of constant mean curvature (CMC) surfaces -- critical points for area under fixed enclosed volume constraint -- lies at the intersection of differential geometry, geometric analysis, integrable systems, and mathematical physics.  They arise naturally in isoperimetric problems. 
This motivates the construction and explicit computation of area and enclosed volume of such surfaces, in particular, if the underlying surfaces has non-trivial topology. In space forms, the construction of CMC surfaces are governed by integrable PDEs, such as the {\em sinh}-Gordon equation. This integrability can be encoded in terms of flatness of an associated family of connections $\nabla^\lambda$ parametrized by a spectral parameter
 $\lambda\in\mathbb C^*,$  and solutions admit many conserved quantities.

It is well known that the area, or more precisely the Willmore energy, is such a conserved quantity. It can be computed from the gauge class of $\nabla^\lambda$ only, which we view as a complex curve into the moduli space of flat connections satisfying a reality condition, see \cite{Hi} for the case of tori and 
 \cite[Theorem 8]{He3}, \cite[Corollary 18]{HHT1} for higher genus surfaces. 

The computation of the enclosed volume is less known. In \cite{Hi2} Hitchin showed that the enclosed volume and energy of minimal tori in $\mathbb S^3$ are related by the holonomy of the natural connection on the Chern-Simons bundle over the moduli space of flat unitary connections along the curve given by the gauge classes of $\nabla^\lambda$. Thus also the enclosed volume is a conserved quantity which can be computed from $[\nabla^\lambda]$ only. The reason this has not been applied to higher genus surfaces was that there has been no effective construction of the associated family when the fundamental group of the surface is non-abelian.

In this paper we first generalize Hitchin's work slightly to compact CMC surfaces in the 3-sphere. The main result is then that the formula can in fact be applied to explicitly compute enclosed areas of higher genus CMC surfaces in the 3-sphere. In other words, we show that the holonomy of the Chern-Simons connection can be independently computed.

The paper is organized as follows: section \ref{sec:csss} reviews the Chern-Simons gauge theory, making the exposition self-contained for readers working on surface theory. Section \ref{sec:evs3}
examines the enclosed volume of CMC surfaces in $\mathbb S^3$,
and its relation to the Willmore energy, and the holonomy of a natural connection on the Chern-Simons bundle. Moreover, we give first examples where the holonomy and thus the enclosed volume can be computed. 
Section \ref{sec:dpw} revisits the integrable surface theory for CMC surfaces,  and derives the holonomy of the Chern-Simons connection in terms of the residues of a Fuchsian DPW potential $\eta^\lambda$, i.e, a particular lift of $[\nabla^\lambda]$ to the space of flat connections with trivial holomorphic structure and pole-like singularities in $\Sigma$. 
Finally,
 we compute in section \ref{section:Lawson} the Taylor expansion for the enclosed volume of CMC deformations of the Lawson minimal surfaces $\xi_{g,1}$ 
(Theorem \ref{theorem:volume-order2}) at $g= \infty$ up to second order. In fact we give two ways of computing it. One using the desingularizing gauges for general Fuchsian DPW potentials introduced in section \ref{sec:dpw}, one more conceptional but restrictive by using Darboux coordinates of the moduli space of Fuchsian systems on the  4-punctured sphere. The formulas obtained from the two approaches look different, but of course they yield the same result at the end.

\section*{Acknowledgements}
LH and SH are supported by the Beijing Natural Science Foundation IS23002 (LH) and IS23003 (SH).\\
MT is supported by the French ANR project Min-Max (ANR-19-CE40-0014).\\

\section{The Chern-Simons line bundle and its connection}\label{sec:csss}
Chern-Simons theory 
was introduced in \cite{Chern-Simons} and later taken up as a tool in quantum field theories \cite{Witten89}, see also \cite{Freed} and references therein.
The aim of this preliminary section is to introduce the Chern-Simons line bundle over the moduli space of flat connections on
a compact Riemann surface $\Sigma$ 
 following \cite{RSW} .
For more details  see also \cite{Freed,Hi2}.

To keep the presentation as concrete as possible, we restrict to
the case of flat $\mathrm{G}$-connections $\nabla$, where $\mathrm{G}$ is either $\mathrm{SU}(2)$ or $\mathrm{SL}(2,\C)$. Here, a $\mathrm{G}$-connection is given by $\nabla=d+A$ on the trivial $\C^2$-bundle over $\Sigma$,
where $A\in\Omega^1(\Sigma,\mathfrak g)$ is the connection 1-form with values in the Lie algebra $\mathfrak g$ of the Lie group $\mathrm{G}.$ The moduli space $\mathcal M_\mathrm{G}$ of flat $\mathrm{G}$-connections is obtained by
quotient of the space of flat $\mathrm{G}$-connections modulo gauge transformations.
It turns out that $\mathcal M(\Sigma)=\mathcal M_{\mathrm{SU}(2)}$  is a compact Hausdorff space with a smooth manifold structure
at the locus of gauge classes of irreducible connections.
For details about moduli spaces, we refer to \cite{AB}. 

\subsection{The Chern-Simons line bundle}

Consider the Maurer-Cartan form
$\omega=g^{-1}dg\in\Omega^1(\mathrm{SL}(2,\C),\mathfrak{sl}(2,\C))$
of  $\mathrm{SL}(2,\C)$.

\begin{Pro}\label{Pro:vol-int}
Let $M$ be a closed 3-manifold, and $f\colon M\to \mathrm{SL}(2,\C).$ Then
\[\frac{1}{12\pi}\int_M\tr(f^*\omega\wedge f^*\omega\wedge f^*\omega)\in2\pi\Z.\]
\end{Pro}
\begin{proof}
Via QR decomposition, $f$ is homotopic to some $\wt f\colon M\to\mathrm{SU}(2).$
Let $f_t$ be any deformation of $f$. Let $X=f_t^{-1}\dot f_t$ where the dot denotes derivative with respect to $t$. Then
$dX=-f_t^*\omega X+f_t^{-1}d\dot f_t.$
Recall that 
$\tr(\alpha\wedge\beta)=(-1)^{k\ell}\tr(\beta\wedge\alpha)$ if $\alpha$, $\beta$ are matrix-valued differential forms of degree $k$, $\ell$.
We compute
\begin{equation*}
\begin{split}
\frac{d}{dt}\int_M\tr\left(f_t^*\omega\wedge f_t^*\omega\wedge f_t^*\omega\right)&=
3\int_M\tr\left(\frac{d}{dt}(f_t^{-1}df_t)\wedge f_t^{-1}df_t\wedge f_t^{-1}df_t\right)\\
&=3\int_M\tr\left((-X f_t^*\omega+f_t^{-1}d\dot f_t)\wedge f_t^*\omega\wedge f_t^*\omega\right)\\
&=3\int_M\tr\left((-f_t^*\omega X +f_t^{-1}d\dot f_t)\wedge f_t^*\omega\wedge f_t^*\omega\right)=3\int_Md\tr(X f_t^*\omega\wedge f_t^*\omega)=0\\
\end{split}
\end{equation*}
by using $\omega\wedge\omega=-d\omega$  and Stokes Theorem.
Hence the integral is constant along smooth deformations, and it suffices to prove Proposition \ref{Pro:vol-int}
for $f\colon M\to\mathrm{SU}(2)$.
Restrict $\omega$ to be the Maurer-Cartan form of $\mathrm{SU}(2)$.
A direct computation shows that 
$-\frac{1}{12}\tr(\omega\wedge\omega\wedge\omega)$
is the volume form of $\mathrm{SU}(2)$ with its round metric of curvature 1 and total volume $2\pi^2.$ Hence,
\[\frac{1}{12\pi}\int_M\tr(f^*\omega\wedge f^*\omega\wedge f^*\omega)=-2\pi\deg(f)\in2\pi\Z.\]
\end{proof}

Given an oriented 3-manifold $M$ and a $\mathrm{SL}(2,\C)$-connection $d+\hat A$ over $M$ consider 
the Chern-Simons functional
\[\mathrm{\mathrm{CS}}(d+\hat A):=\mathrm{\mathrm{CS}}_M(d+\hat A):=\frac{1}{4\pi}\int_M\tr(\hat A\wedge d\hat A+\tfrac{2}{3}\hat A\wedge \hat A\wedge \hat A).\]
Let $\Sigma$ be a compact oriented surface.
We define (the cocyle) $\Theta$ for a $\mathrm{SL}(2,\C)$-connection $d+A$ and a gauge transformation $g\colon\Sigma\to\mathrm{SL}(2,\C)$ on the Riemann surface  $\Sigma$ by
\[\Theta(d+A,g):=\exp\left (i \,\mathrm{CS}_B((d+\hat A).\hat g)-i\,\mathrm{CS}_B(d+\hat A)\right),\]
where  $B$ is any compact oriented 3-manifold with oriented boundary $\partial B=\Sigma$
and $(\hat A,\hat g)$
are arbitrary (smooth) extensions of $(A,g)$ on $\Sigma=\partial B$ to $B.$ Because  $d+\hat A$ is not assumed to be flat on $B$, such
extensions  always exist.
We will also write 
\[\Theta(d+A,g)=\Theta(A,g)\]
for short.
Moreover, we will denote the Chern-Simons integrant by
\[cs(d+\hat A)=\tr\left(\hat A\wedge d\hat A+\tfrac{2}{3}\hat A\wedge \hat A\wedge \hat A\right),\]
such that
\[\Theta(d+A,g)=\exp\left ({\frac{i}{4\pi}} \int_B{cs}((d+\hat A).\hat g)-\frac{i}{ 4\pi}\int_B{cs}(d+\hat A)\right)\]
for some extensions $(\hat A,\hat g)$ of $(A, g)$ to $B$ as above.
The following proposition is classical, see \cite{RSW,Freed,Hi2}.
\begin{Pro}\label{pro:chs} Let $M$ be a  3-manifold, and $A$ be a 
$\mathrm{SL}(2,\C)$-connection on $M$.
Then we have for every gauge transformation $g\colon M\to \mathrm{SL}(2,\C)$
that 
\begin{equation}\label{eq:CSdidd}cs((d+A).g)-cs(d+A)=-\tfrac{1}{3}\tr(g^*\omega\wedge g^*\omega\wedge g^*\omega)+d\tr(g^{-1}Ag\wedge g^*\omega).\end{equation}
 In particular, if $M$ is closed without boundary we have (by Stokes Theorem and Proposition \ref{Pro:vol-int})
 \[\mathrm{CS}((d+A).g)-\mathrm{CS}(d+A)\in 2\pi \Z.\] 
\end{Pro}
\begin{proof}
Let
$\wt A:=g^{-1}Ag+g^{-1}dg$
and $\theta:=g^*\omega=g^{-1}dg$. Then we compute
\[\begin{split}
cs(d + \wt A)=&\tr\left( \wt A\wedge  d\wt A+\tfrac{2}{3} \wt A\wedge \wt A\wedge \wt A\right)\\
=&\tr\left((g^{-1}Ag+\theta)\wedge (-\theta\wedge g^{-1}Ag+g^{-1}dA\,g - g^{-1}Ag\wedge\theta-\theta\wedge\theta)\right)\\
&+\tfrac{2}{3}\tr\left((g^{-1}Ag+\theta)\wedge (g^{-1}Ag+\theta)\wedge (g^{-1}Ag+\theta)\right)\\
=&\tfrac{2}{3}\tr\left(g^{-1}Ag\wedge g^{-1}Ag\wedge g^{-1}Ag\right)
+\tr\left(g^{-1}Ag\wedge g^{-1}Ag\wedge\theta\right)\left(-1-1+\tfrac{2}{3}+\tfrac{2}{3}+\tfrac{2}{3}\right)\\
&+\tr\left(g^{-1}Ag\wedge\theta\wedge\theta\right)\left(-1-1-1+\tfrac{2}{3}+\tfrac{2}{3}+\tfrac{2}{3}\right)
+\tr\left(\theta\wedge\theta\wedge\theta\right)\left(-1+\tfrac{2}{3}\right)\\
&+\tr\left(g^{-1}Ag\wedge g^{-1}dA\,g\right)
+\tr\left(\theta\wedge g^{-1} dA\,g\right)\\
=&\tr(A\wedge dA )+\tfrac{2}{3}\tr(A\wedge A\wedge A)-\tfrac{1}{3}\tr(\theta\wedge \theta\wedge \theta)-\tr(g^{-1}Ag\wedge \theta\wedge\theta)+\tr(\theta\wedge g^{-1}dA g)
\end{split}\]
Now using $d\theta+\theta\wedge\theta=0$, we have
\[d\tr\left(g^{-1}Ag\wedge\theta\right)=-\tr\left(g^{-1}Ag\wedge\theta\wedge\theta\right)+\tr\left(g^{-1}dA\, g\wedge\theta\right),\]
from which Proposition \ref{pro:chs} follows.
\end{proof}

\begin{Cor}
Let $\Sigma$ be a closed oriented surface. Moreover, let $A$ be a $\mathrm{SL}(2,\C)$ connection over $\Sigma$ and $g\colon \Sigma\to\mathrm{SL}(2,\C)$ a gauge transformation.
Then, the cocycle $\Theta(A,g)$ does neither depend on the choice of $B$ with $\partial B=\Sigma$  nor on the extension $(\hat A,\hat g)$ of
$(A,g)$ on $\Sigma=\partial B$ to $B$.
\end{Cor}

Let  $\mathcal A_f$ denote the (infinite dimensional) space of flat $\mathrm{SL}(2,\C)$ connections on the trivial $\C^2$ bundle over the compact oriented surface $\Sigma$, and $\mathcal A_f^u$ denote the subspace of flat $\mathrm{SU}(2)$ connections (with respect to the standard hermitian metric).
Note that there exist some hermitian metric for which a flat connection $\wt \nabla$ is unitary if and only if it is of the form $\wt \nabla = \nabla.g$ for some $\nabla\in \mathcal A_f^u$ and
some complex gauge transformation $g\colon \Sigma\to\mathrm{SL}(2,\C).$ 
We call such flat connections $\wt\nabla$ unitarizable and denote the space of flat unitarizable connections by $\mathcal A_f^{ua}$.
Denote the group of all complex gauge transformations by
$\mathcal G^\C,$ and consider the space $\mathcal A_f^{ua}$ as a principal $\mathcal G^\C$ bundle over the moduli space $\mathcal M(\Sigma) = \mathcal A_f^u/\mathcal G^\C$ 
of flat $\mathrm{SU}(2)$-connections.
The moduli space $\mathcal M(\Sigma) $  is Hausdorff, and smooth away from reducible gauge classes.

\begin{Pro}[\cite{RSW}]\label{pro:csbundle}
On
$\mathcal A_f^{ua}\times \C$ define an equivalence relation by
\[ (d+A,\mu)\sim ((d+A).g,\Theta(d+A,g)\mu)\]
with $g\in\mathcal G^\C$. The corresponding complex line bundle
\[\mathcal L=(\mathcal A_f^{ua}\times \C)/\mathcal G^\C\to\mathcal M(\Sigma)=\mathcal A_f^{ua}/\mathcal G^\C\]
is called the Chern-Simons bundle and $\mathcal L$ inherits a hermitian metric.
\end{Pro}
\begin{proof} 
As $\mathcal G^\C$ does not act freely on $\mathcal A_f^{ua},$ we  need to check that
$\Theta(A,g)=1$
whenever 
$(d+A).g=d+A.$
We distinguish different cases by the complexity of their monodromy.
\begin{enumerate}
\item If $d+A$ is irreducible, then every $\mathrm{SL}(2,\C)$ gauge transformation preserving $d+A$ must satisfy $g=\pm 1$. Then, we can choose the extension $\hat g=\pm 1$ to be constant as well, which satisfies  
$(d+\hat A).\hat g =d+\hat A $  for any extension of $d+A$ to $B$.  Thus by definition $\Theta(A,g)=1.$
\item If $d+A$ is trivial, there exists a gauge $h$ such that $(d+A).h=d.$ Let $\hat h\colon B\to\mathcal \mathrm{SL}(2,\C)$ be an extension of $h$ to $B$. Then
$\hat A=-d\hat h \hat h^{-1}$ is an extension of $A$ to $B$. Let $g$ be a gauge satisfying $(d+A).g=d+A.$
Then $h^{-1}gh=C$ is constant.  Define
\[\hat g=\hat hC \hat h^{-1}.\]
Then 
\[(d+\hat A).\hat g=(d.{\hat h}^{-1}).\hat hC \hat h^{-1}= d.\hat h^{-1}=(d+\hat A),\] hence $\Theta(A,g)=1$ holds in this case as well.
\item Next, we consider the case where  $d+A$ is $\mathbb Z_2$, i.e., $d+A$ is not trivial but all monodromies are $\pm 1$.  As $(-1)^2=1$, the monodromy is abelian (and unitarizable). Therefore, $d+A$ splits into the direct sum of line bundle connections. Moreover, on a 2-fold covering of $\Sigma$ the connection $d+A$ is gauge equivalent to the trivial connection, and
we can find some 3-dimensional handlebody $B$ with $\partial B=\Sigma$ such that the $\mathbb Z_2$ connection extends to a flat connection $d+\hat A$ on $B$ (this means that we chose
the cycles $\beta_1,\dots, \beta_{\text{genus}(\Sigma)}$ of $\Sigma$ which are  contractible in $B$ in such a way, that the monodromy along $\beta_k$ is  1 for all $k=1,\dots,\text{genus}(\Sigma)$).
Then $d+\hat A$ is gauge equivalent by some $h$ on $B$ to \[d+\tfrac{1}{2}\begin{pmatrix}\frac{df}{f}&0\\0&-\frac{df}{f}\end{pmatrix}\]
for some function $f\colon B\to\C,$
and 
\[(\hat h_{\mid\Sigma})^{-1}g\,\hat h_{\mid\Sigma}=\begin{pmatrix} a& b/f_{\mid\Sigma}\\ c f_{\mid\Sigma}&d\end{pmatrix} \quad \text{with} \quad ad-bc=1.\]
In this way we can extend $g$ to $\hat g$ satisfying $(d+\hat A).\hat g=d+\hat A$, which shows $\Theta(A,g)=1$.
\item If $d+A$ is reducible but not $\mathbb Z_2$, then we can split $\C^2=L\oplus L^*$ into its two unique parallel subbundles (since $d+A$ is unitarizable). Thus there exists a  gauge transformation $h$ on $\Sigma$ such that
\[d+\wt A:=(d+A).h=d+\begin{pmatrix} \eta&0\\0&-\eta\end{pmatrix}.\]
Let $g$ be a gauge with $(d+A).g = d+A$. Then
$(d+\wt A)=(d+\wt A).(h^{-1}gh)$
implying  $h^{-1}gh=\mathrm{diag}(c,\tfrac{1}{c})$
 is diagonal and constant (here we use that  $d+A$ and therefore $d+\wt A$ are reducible but not $\mathbb Z_2$).
Extend $h$ on $\Sigma$ to $\hat h$ on $B$ (which is possible as $h$ maps into $\mathrm{SL}(2,\C)$), and extend $\eta\in\Omega^1(\Sigma,\C)$ to $\hat\eta\in\Omega^1(B,\C)$ using a bump function. Consider the extensions
\[d+\hat A=\left(d+\begin{pmatrix} \hat\eta&0\\0&-\hat\eta\end{pmatrix}\right).\hat h^{-1}\]
and 
\[\hat g=\hat h\begin{pmatrix}c&0\\0&\tfrac{1}{c}\end{pmatrix}\hat h^{-1}.\]
Then
\[(d+\hat A).\hat g=(d+\hat A),\] and $\Theta(A,g)=1$ also holds in the last case.
\end{enumerate}
To show that $\sim$ is an equivalence relation we only need to  show that $\sim$ is transitive, which is equivalent to the cocyle relation
\[\Theta(A,g)\Theta(A.g,h)=\Theta(A,gh)\]
for $A\in\mathcal A_f^{ua}$ and $g,h\in\mathcal G^\C.$ This follows directly from the
definition and
\[\mathrm{CS}((d+A).(gh))-\mathrm{CS}(d+A)=\mathrm{CS}((d+A).(gh))-\mathrm{CS}((d+A).g)+\mathrm{CS}((d+A).g)-\mathrm{CS}(d+A).\]

That the line bundle $\mathcal L$ inherits a hermitian metric follows from the act that $\mathrm{CS}(d+A)\in \R$ for $A\in\mathcal A_f^u.$ Thus there is a well-defined hermitian metric on $\mathcal L$ for which
a vector $v\in\mathcal L$ has length 1 if and only if it can be represented as
$(d+A,\mu)$
with $d+A$ unitary and $\mu\in S^1.$ 
\end{proof}

\subsection{The natural connection $\mathcal D$ on the Chern-Simons line bundle}
Let $\mathcal A$ be the (affine) space of $\mathrm{SL}(2,\C)$ connections
which we identify with $\Omega^1(\Sigma,\mathrm{sl}(2,\C)).$
Consider the trivial line bundle $\mathcal A\times \C\to \mathcal A$ and the connection
\begin{equation}\label{def:conA}d+\hat \alpha=d_\mathcal A+\hat \alpha,\end{equation}
where 
\[\hat\alpha_{d_\Sigma+A}(X):=\frac{i}{4\pi}\int_\Sigma\tr(A\wedge X).\]
The curvature of  $d+\hat \alpha$ is
\begin{equation}\label{cur-CS}
\begin{split}
d\hat\alpha_A(X,Y)=& X\cdot \hat\alpha_{d+A}(Y)-Y\cdot \hat\alpha_{d+A}(X)-\hat\alpha_{d+A}([X,Y])
=\frac{i}{2\pi}\int_\Sigma\tr(X\wedge Y)=:\Omega(X,Y).
\end{split}
\end{equation}
By Atiyah-Bott \cite{AB}, 
 $\Omega$ restricted to $\mathcal A_f^u$ is the pull-back of a (imaginary valued) symplectic form (also denoted by $\Omega$) on $\mathcal M(\Sigma)$. The cohomology class of $\Omega$ on $\mathcal M(\Sigma)$ is integral, and
 there is a natural connection $\mathcal D$ on $\mathcal L\to\mathcal M(\Sigma)$ whose curvature is $\Omega.$
We explain this construction in the following, see also \cite{RSW} :
 \begin{Pro}\label{pro:4}
 The connection $d+ \hat\alpha$ restricted to $\mathcal A_f^{ua}$ is the pull-back of a connection $\mathcal D$ on $\mathcal L\to\mathcal M(\Sigma)$.
 \end{Pro}
\begin{proof}
Consider a curve $\tau\in I\mapsto A(\tau)\in\mathcal A_f^{ua},$ and a family of gauge transformations $\tau\in I\mapsto g(\tau)$.
Let $\wt A(\tau)=A(\tau).g(\tau)$.
Then $d+A(\tau)$ and $d+\wt A(\tau)$  give
rise to the same curve $\gamma(\tau)$ in the moduli space $\mathcal M,$ but they determine different trivializations of $\mathcal L$ over $\gamma.$
Since the transition function between these two trivializations is
$\Theta(A(\tau),g(\tau))$, we have to prove that
\begin{equation}\label{eq:Dtra}\hat\alpha_{d+A}(\dot{A})=\hat\alpha_{d+\wt A}(\dot{\tilde A})+\frac{d}{d\tau}\log\Theta(A(\tau),g(\tau))\mid_{\tau=0}.\end{equation}

This then implies that the  expressions in different local trivializations fit together, yielding a well-defined connection $\mathcal D.$
We first compute the last term in \eqref{eq:Dtra}.
\begin{Lem}\label{lem:dertheta}
We have 
\begin{equation}\label{lem:dertheta}
\begin{split}
\frac{d}{d\tau}\log(\Theta(A(\tau),g(\tau)))\mid_{\tau=0}
&=\tfrac{i}{4\pi}\int_\Sigma\tr( g^{-1}dg\wedge\xi g^{-1}dg)+\tfrac{i}{4\pi}\int_\Sigma\tr( g^{-1}A\wedge d\dot g)\\
&-\tfrac{i}{4\pi}\int_\Sigma\tr( dg\wedge g^{-1}\dot A )+\tfrac{i}{4\pi}\int_\Sigma\tr( dg\wedge\xi g^{-1}A).
\end{split}
\end{equation}
\end{Lem}
\begin{proof}
We denote derivative with respect to $\tau$ at $\tau=0$ by a dot and let
$\xi=g^{-1}\dot{g}$.
Using \begin{equation}
\begin{split}
d\xi&=-g^{-1}dg\xi+g^{-1}d\dot g,\\
 \frac{d}{d\tau}|_{\tau=0}(g^{-1}dg)&=-\xi g^{-1}dg+g^{-1}d\dot g,\\
\tr(\xi g^{-1}dg\wedge g^{-1}dg\wedge g^{-1}dg)&=\tr(g^{-1}dg\xi\wedge g^{-1}dg\wedge g^{-1}dg)\\
\end{split}
\end{equation}
 we obtain from  \eqref{eq:CSdidd}
\begin{equation}
\begin{split}
\frac{d}{d\tau}|_{\tau=0}\log(\Theta(A(\tau),g(\tau)))=&\frac{i}{4\pi}
\frac{d}{d\tau}|_{\tau=0}\left(-\frac{1}{3}\int_B \tr(g^{-1}dg\wedge g^{-1}dg\wedge g^{-1}dg)+\int_\Sigma\tr(g^{-1}A\wedge dg)\right)\\
=&-\frac{i}{4\pi}\int_B \tr \left( \frac{d}{d\tau}{|_{\tau=0}}(g^{-1}dg)\wedge g^{-1}dg\wedge  g^{-1}dg \right)+\frac{i}{4\pi}\int_\Sigma\tr(-\xi g^{-1}A\wedge dg)\\&+\frac{i}{4\pi}\int_\Sigma\tr(g^{-1}\dot A\wedge dg)+\frac{i}{4\pi}\int_\Sigma\tr(g^{-1}A\wedge d\dot g )\\
=&-\frac{i}{4\pi}\int_\Sigma \tr( \xi g^{-1}dg\wedge  g^{-1}dg)+\frac{i}{4\pi}\int_\Sigma\tr(dg\wedge \xi g^{-1}A)\\
&-\frac{i}{4\pi}\int_\Sigma\tr(dg\wedge g^{-1}\dot A)+\frac{i}{4\pi}\int_\Sigma\tr(g^{-1}A\wedge d\dot g).
\end{split}
\end{equation}
\end{proof}
To prove Proposition \ref{pro:4}, we compute
\begin{equation*}
\begin{split}
\wt A&=g^{-1} A g + g^{-1}dg\\
\dot{\tilde A}&=-\xi g^{-1}Ag+g^{-1}\dot A g+g^{-1}Ag \xi-\xi g^{-1}dg+g^{-1}d\dot g.
\end{split}
\end{equation*}
By Lemma \ref{lem:dertheta}, we then obtain
\begin{align*}
\hat\alpha_{d+\wt A}&(\dot{\tilde A})+\frac{d}{d\tau}\log\Theta(A(\tau),g(\tau))
=\frac{i}{4\pi}\int_{\Sigma}\tr\left(\wt A\wedge\dot{\tilde A}\right)+\frac{d}{d\tau}\log\Theta(A(\tau),g(\tau))\\
&=\frac{i}{4\pi}\int_{\Sigma}\tr\left(A\wedge \dot{A}+2 g^{-1} A\wedge A \dot{g}+2 g^{-1}dg\wedge g^{-1} A \dot{g}+2 g^{-1}A\wedge d\dot{g}+g^{-1} dg\wedge g^{-1}d\dot{g}\right)\\
&=\frac{i}{4\pi}\int_{\Sigma}\tr\left(A\wedge\dot{A}\right)+\frac{i}{4\pi}\int_{\Sigma}d\,\tr\left(-2g^{-1}A\dot{g}-g^{-1}d\dot{g}\right)=\hat\alpha_{d+A}(\dot{A})
\end{align*}
using flatness ($dA+A\wedge A=0$) and Stokes Theorem.
\end{proof}
A priori, the connection $\mathcal D$ could behave badly at singular points of $\mathcal M(\Sigma)$ (these are given by reducible flat connections $\nabla$ on $\Sigma$). The following lemma shows that this is not the case.
\begin{Lem}\label{Lem:non-singular}
The local monodromy of $\mathcal D$ on $\mathcal L\to\mathcal M$ around a non-smooth point $[\nabla]\in\mathcal M$ is trivial, i.e.
\[\Hol(\gamma)=\exp\left(-\int_D \Omega\right) \]
where $\Omega$ is the curvature of $\mathcal D$ and $D\subset \mathcal M$ is an arbitrary  disc with (oriented) boundary $\gamma.$
\end{Lem}
\begin{proof}
The statement follows since the Chern-Simons line bundle $\mathcal L$ and the connection 1-forms
extend continuously to the reducible points by the proof of Proposition \ref{pro:csbundle}.
\end{proof}

\section{Enclosed volume of CMC surfaces in $\mathbb S^3$}\label{sec:evs3}
\begin{Def}
Let $f\colon \Sigma\to\mathbb S^3$ 
be an embedded closed oriented surface  in the round 3-sphere of curvature 1.
It bounds a domain $B\subset \mathbb S^3$, and
let 
\[\mathcal V(f):=Vol_{\mathbb S^3}(B)\in (0,2\pi^2)\]
denote the enclosed volume of $f$.
If $f\colon \Sigma\to\mathbb S^3$ is an immersion, we consider any extension $F\colon B\to \mathbb S^3=\mathrm{SU}(2)$
and define the (algebraic) enclosed volume by
\[\mathcal V(f):=\int_BF^*\text{vol}_{\mathbb S^3}=-\frac{1}{12}\int_BF^*\tr(\omega\wedge\omega\wedge\omega)\in \R/2\pi^2\Z.\]
\end{Def}
By Proposition \ref{Pro:vol-int}, $\mathcal V(f)$ modulo $2\pi^2\Z$ does not depend on
the choice of $B$ or $F$.
\begin{Rem}
A closed surface $f\colon \Sigma\to\mathbb S^3$ is called Alexandrov embedded if there is a handlebody  $B$ bounded by $\Sigma$ such that $f$ extends 
to an immersion $F\colon B \to\mathbb S^3$. Hence, also for Alexandrov embedded surfaces, the enclosed volume is a well-defined real number.
\end{Rem}
\subsection{The associated family of flat connections}
Similarly to minimal surfaces,  CMC surfaces $f:\Sigma\to\mathbb S^3$ can be described by an associated family of flat connections.
This fact can be seen as a corollary of the Lawson correspondence between minimal and CMC surfaces in $\mathbb S^3.$ Again, we identify
$\mathrm{SU}(2)=\mathbb S^3$
via
\[g=\begin{pmatrix} x&-\bar y\\ y&\bar x\end{pmatrix}\in\mathrm{SU}(2) \mapsto (x,y)\in \mathbb S^3\subset\C^2.\]
Note that the round metric of curvature 1 is given by 
\[X,Y\in\mathfrak{su}(2)\cong T_g\mathrm{SU}(2)\longmapsto -\tfrac{1}{2}\tr(XY),\]
where we use the left trivialization of the tangent bundle.

The following result is well known, see  \cite{Hi} for the minimal case, and \cite[Lemma 2.2]{SKKR} for the CMC case, or \cite{B, HHSch}. We include the proof here to make the paper self-contained.
\begin{The}
\label{theorem:CMC-connection}
Let $(\nabla^{\lambda})_{\lambda\in\C^*}$ be a family of connections on $\Sigma\times\C^2$ of the form
$$\nabla^{\lambda}=\nabla+\lambda^{-1}\Phi-\lambda\Phi^*$$
where $\nabla$ is a unitary connection
and $\Phi$ is a nilpotent $\sl(2,C)$-valued 1-form of type $(1,0)$ on the Riemann surface $\Sigma$.
Assume that $\nabla^{\lambda}$ is flat for all $\lambda$, and assume there exists
distinct points $\lambda_1,\lambda_2\in\S^1$ (called the Sym points) such that $\nabla^{\lambda_1}$
and $\nabla^{\lambda_2}$ are trivial. Then the gauge $f\colon\Sigma\to SU(2)$ such that
$$\nabla^{\lambda_1}. f=\nabla^{\lambda_2}$$
is a conformal CMC immersion of mean curvature
\begin{equation}
\label{eq:H}
H=i\frac{\lambda_1+\lambda_2}{\lambda_1-\lambda_2}.
\end{equation}
Its Willmore energy is given by
$$\mathcal W(f)=2i\int_{\Sigma}\tr(\Phi\wedge\Phi^*).$$
Conversely, if $f:\Sigma\to SU(2)$ is a conformal CMC immersion, let $\Psi$ be
the 1-form of type $(1,0)$ defined by
$$f^{-1}df=\Psi-\Psi^*.$$
Choose two points
$\lambda_1,\lambda_2\in\S^1$ such that \eqref{eq:H} holds and define
$$\Phi=\frac{1}{\lambda_2^{-1}-\lambda_1^{-1}}\Psi$$
$$\nabla^{\lambda}=d+(\lambda^{-1}-\lambda_1^{-1})\Phi-(\lambda-\lambda_1)\Phi^*.$$
Then $\nabla^{\lambda}$ is  flat for all $\lambda\in\C^*$, unitary for $\lambda\in\S^1$ and trivial at
$\lambda_1$ and $\lambda_2$.
\end{The}
\begin{proof}
Fix a trivialization such that $\nabla^{\lambda_1}=d$, i.e., 
$$\nabla^{\lambda}=d+(\lambda^{-1}-\lambda_1^{-1})\Phi-(\lambda-\lambda_1)\Phi^*.$$
Then $\nabla^{\lambda_2}=d+f^{-1}df$ and thus
$$f^{-1}df=(\lambda_2^{-1}-\lambda_1^{-1})\Phi-(\lambda_2-\lambda_1)\Phi^*.$$
Using $\Phi^2=0$, the induced metric on $\Sigma$ is
$$g=\frac{-1}{2}\tr(f^{-1}df\otimes f^{-1} df)
=\frac{1}{2}|\lambda_2-\lambda_1|^2\tr\left(\Phi\otimes\Phi^*+\Phi^*\otimes\Phi\right)$$
which shows that $f$ is conformal.

In the following, we switch back and forth between (unitary) connections on the spinor bundle $V=\underline\C^2$ and the associated connections
on the associated euclidean $\mathfrak{su}(2)$-bundle, and the complexification thereof. For details in the minimal case see \cite{He1}.
The Levi-Civita connection is
$\nabla^{LC}=d+\frac{1}{2}f^{-1}df,$
i.e.,
$$f^{-1}df\left(\nabla^{LC}_X Y\right)=d_X f^{-1}df(Y)+\frac{1}{2}\left [f^{-1}df(X),f^{-1}df(Y)\right].$$
for $X,Y$ vector fields on $\Sigma$.
Flatness of $\nabla^{\lambda}$ for all $\lambda$ gives us
$$0=F^{\nabla^{\lambda}}=(\lambda^{-1}-\lambda_1^{-1})d\Phi
-(\lambda-\lambda_1)d\Phi^*
-(\lambda^{-1}-\lambda_1^{-1})(\lambda-\lambda_1)[\Phi\wedge\Phi^*].$$
Looking at the coefficients of $\lambda^{-1}$ and $\lambda$ we obtain
\begin{equation}
\label{eq:dPhi}
\begin{cases}
d\Phi=-\lambda_1[\Phi\wedge\Phi^*]\\
d\Phi^*=\lambda_1^{-1}[\Phi\wedge\Phi^*].\end{cases}
\end{equation}
Let $\ast$ be the star operator on $\Sigma$ (i.e., $\ast dz=i dz$). Then
$$2H\, f^{-1}N\, dA= - d^{\nabla^{LC}} \ast f^{-1}df.$$
On the one hand,
\begin{equation}
\label{eq:NdA}
2 f^{-1}N\, dA=\frac{1}{2}[f^{-1}df\wedge f^{-1}df]
=\frac{(\lambda_2-\lambda_1)^2}{\lambda_1\lambda_2}[\Phi\wedge\Phi^*],
\end{equation}
while on the other hand we get from
$\ast f^{-1}df=i(\lambda_2^{-1}-\lambda_1^{-1})\Phi+i (\lambda_2-\lambda_1)\Phi^*$
\begin{align}
\nonumber
d^{\nabla^{LC}} \ast f^{-1}df&=d\ast f^{-1}df+\frac{1}{2}[f^{-1}df\wedge \ast f^{-1}df]\\
\label{eq:Deltaf}
&=i(\lambda_2^{-1}-\lambda_1^{-1})d\Phi+i(\lambda_2-\lambda_1)d\Phi^*
+\frac{i}{2}(\lambda_2^{-1}-\lambda_1^{-1})(\lambda_2-\lambda_1)\left([\Phi\wedge\Phi^*]-[\Phi^*\wedge\Phi]\right)\\
\nonumber
&=i\frac{(\lambda_2^2-\lambda_1^2)}{\lambda_1\lambda_2}[\Phi\wedge\Phi^*].
\end{align}
Hence $f$ has constant mean curvature $H$ given by \eqref{eq:H}.
Its area element is
$$dA=\frac{i}{2}|\lambda_2-\lambda_1|^2\tr(\Phi\wedge\Phi^*).$$
Equation \eqref{eq:H} gives
$$1+H^2=\frac{4}{|\lambda_2-\lambda_1|^2}.$$
Hence, the Willmore energy is given by
$$\mathcal W(f)=\int_{\Sigma}(1+H^2)dA=2i\int_{\Sigma}\tr(\Phi\wedge\Phi^*).$$
Conversely, if $f$ has constant mean curvature $H$ given by \eqref{eq:H}, define
$\Phi$ and $\nabla^{\lambda}$ as in Theorem \ref{theorem:CMC-connection}.
Then
$$f^{-1}df=(\lambda_2^{-1}-\lambda_1^{-1})\Phi-(\lambda_2-\lambda_1)\Phi^*$$
as before.
Conformality of $f$ gives that $\Phi$ is nilpotent.
Moreover, $\nabla^{\lambda_1}=d$ and
$\nabla^{\lambda_2}=d+f^{-1}df$ are trivial.
The Maurer-Cartan equation
$d(f^{-1}df)+f^{-1}df\wedge f^{-1}df=0$
gives the equation
$$(\lambda_2^{-1}-\lambda_1^{-1})d\Phi-(\lambda_2-\lambda_1)d\Phi^*
=(\lambda_2^{-1}-\lambda_1^{-1})(\lambda_2-\lambda_1)[\Phi\wedge\Phi^*].$$
By Equations \eqref{eq:NdA} and \eqref{eq:Deltaf}, the fact that $f$ has constant mean curvature $H$ gives the equation
$$(\lambda_2^{-1}-\lambda_1^{-1})d\Phi+(\lambda_2-\lambda_1)d\Phi^*
=-(\lambda_2^{-1}-\lambda_1^{-1})(\lambda_2+\lambda_1)[\Phi\wedge\Phi^*].$$
Solving for $d\Phi$ and $d\Phi^*$ yields \eqref{eq:dPhi} and flatness of 
$\nabla^{\lambda}$ for $\lambda\in\C^*$ follows.
\end{proof}

\subsection{Enclosed volume}
The following theorem extends \cite[Theorem 5]{Hi2} to compact CMC surfaces in $\mathbb S^3$.

\begin{The}\label{thm:holev}
Let $f\colon\Sigma\to\mathrm{SU}(2)=\mathbb S^3$ be a compact CMC immersion with mean curvature $H$ and
associated family of flat connections $\nabla^\lambda$.
Expressing the Sym points as $\lambda_1=e^{i\tau_1}$ and $\lambda_2=e^{i\tau_2}$,
 the holonomy $\Hol(\mathcal D,\gamma)$ of the connection $\mathcal D$ on $\mathcal L\to\mathcal M(\Sigma)$ along the closed loop 
\begin{equation}\label{def:gamma}\gamma\colon [\tau_1,\tau_2]\longrightarrow \mathcal M(\Sigma);\quad \tau \longmapsto [\nabla^{e^{i\tau}}]\end{equation}
is given by
\[\Hol(\mathcal D,\gamma)=\exp\left(\frac{i}{4\pi}\big(\tau_2-\tau_1-\sin(\tau_2-\tau_1)\big)\mathcal W(f)-\frac{i}{\pi}\mathcal V(f)\right).\]
\end{The}
\begin{proof}
By construction $\nabla^\lambda$ is trivial at $\lambda_1$ and $\lambda_2$.
Hence $\gamma\colon [\tau_1,\tau_2]\to \mathcal M(\Sigma);\, \tau \mapsto [\nabla^{e^{i\tau}}]$ is actually a closed curve in the moduli space $\mathcal M(\Sigma)$.
And the holonomy of $\mathcal D$ along $\gamma$ can be computed as follows:
take the (non-closed) lift $\hat\gamma$ of $\gamma$ given by 
$\tau\in[\tau_1,\tau_2]\mapsto \nabla^{e^{i\tau}}.$
Note that $\hat\gamma(\tau_2)$ and $\hat\gamma(\tau_1)$ differ by the gauge $f$. As a consequence, we have in $\mathcal L$
$$(\hat\gamma(\tau_1),X)\sim (\hat\gamma(\tau_2),\Theta(d,f)X)$$
where by Proposition \ref{pro:chs} and the definition of enclosed volume
\[\Theta(d,f)=\exp\left(\tfrac{i}{4\pi}\int_B\tfrac{-1}{3}\tr(\wh f^*\omega\wedge \wh f^*\omega\wedge \wh f^*\omega)\right)=\exp\left(\tfrac{i}{\pi}\mathcal V(f)\right).\]

Let $X(\tau)$ be the parallel transport of $X(\tau_1)=1$ along $\hat\gamma$.
Since $\mathcal L$ is a line bundle
$$X(\tau_2)=\exp\left(-\int_{\tau_1}^{\tau_2}\hat\alpha_{\hat\gamma(\tau)}(\hat\gamma'(\tau))d\tau\right).$$
We have
\begin{equation*}
\begin{split}
\hat\gamma(\tau)&=d+(e^{-i\tau}-e^{-i\tau_1})\Phi-(e^{i\tau}-e^{i\tau_1})\Phi^*\\
\hat\gamma'(\tau)&=-ie^{-i\tau}\Phi-ie^{i\tau}\Phi^*\\
\hat\alpha_{\hat\gamma(\tau)}(\hat\gamma'(\tau))
&=
\frac{i}{4\pi}\int_\Sigma\tr\left(\left((e^{-i\tau}-e^{-i\tau_1})\Phi-(e^{i\tau}-e^{i\tau_1})\Phi^*\right)\wedge\left(-ie^{-i\tau}\Phi-ie^{i\tau}\Phi^*\right)\right)\\
&=(2-2\cos(\tau-\tau_1))\frac{1}{4\pi}\int_\Sigma\tr(\Phi\wedge\Phi^*).
\end{split}
\end{equation*}
Hence,
\[
X(\tau_2)=
\exp\left(\frac{1}{2\pi}\int_\Sigma\tr(\Phi\wedge\Phi^*)\int_{\tau_1}^{\tau_2}(\cos(\tau-\tau_1)-1)d\tau\right)=\exp\left(\frac{i}{4\pi}\big(\tau_2-\tau_1-\sin(\tau_2-\tau_1)\big)\mathcal W(f)\right).\]
But $X(\tau_2)$ is expressed in the trivialization $\hat\gamma(\tau_2)$. To obtain the holonomy, we must express it in the trivialization $\hat\gamma(\tau_1)$, so divide it by
$\Theta(d,f)$.
Overall, the 
 holonomy along $\gamma$ is
 \[\Hol(\mathcal D,\gamma)=\exp\left(\tfrac{i}{4\pi}\big(\tau_2-\tau_1-\sin(\tau_2-\tau_1)\big)\mathcal W(f)\right) \exp\left(\tfrac{i}{\pi}\mathcal V(f)\right)^{-1}\]
 as claimed.
\end{proof}

\subsection{Examples}
\begin{Exa}
In the case of a minimal surface, we can take $\lambda_1=1$ and $\lambda_2=-1$. Theorem \ref{thm:holev} then gives
\[\Hol(\mathcal D,\gamma)=\exp\left(\tfrac{i}{4}\operatorname{Area}(f)-\tfrac{i}{\pi}\mathcal V(f)\right)\]
as in \cite{Hi2}.
\end{Exa}
\begin{Exa}
Consider a CMC sphere in $\mathbb S^3.$
Its Willmore energy is $4\pi$ as the surface is totally umbilic. 
Let $\tau_1=0$ and $\tau_2=\theta\in(0,\pi)$. Then, the corresponding mean curvature is given by
$$H=i\frac{e^{i\theta}+1}{e^{i\theta}-1}=-\cot(\tfrac{\theta}{2})<0.$$
All flat connections on $\mathbb S^2$ are trivial, and the moduli space $\mathcal M(\mathbb S^2)$ consists
of a single point only. In particular, the holonomy of $\mathcal D$ along any curve $\gamma$ is trivial, i.e.,  $\Hol(\mathcal D,\gamma)=1$.
From Theorem \ref{thm:holev} we obtain
\[\mathcal V(f)=\pi(\theta-\sin(\theta))\mod 2\pi^2.\]
On the other hand, we compute
\[\pi(\theta-\sin(\theta))=4\pi\int_{r=0}^{\theta/2}\sin(r)^2dr
=4\pi\int_0^{\cot^{-1}(-H)}\sin(r)^2dr\\
=\operatorname{Vol}(B_H)\]
where $B_H$ is a spherical 3-ball with boundary mean curvature $H<0.$
\end{Exa}
\begin{Exa}
Let $r,s\in\R^{>0}$ satisfying $r^2+s^2=1,$ and consider the conformally parametrized  homogeneous CMC torus
\[f=f_{r,s}\colon \C/(r\Z+is\Z)\to\mathrm{SU}(2);\, [x+iy]\mapsto \left(
\begin{array}{cc}
 r e^{\frac{2 i \pi  x}{r}} & -se^{-\frac{2 i \pi  y}{s}} \\
 s e^{\frac{2 i \pi  y}{s}} & r e^{-\frac{2 i \pi  x}{r}} \\
\end{array}
\right).\]
It has mean curvature
$H=\frac{1-2 s^2}{2 r s},$
and the two Sym points are given by
\[\lambda_1=1\quad\text{and}\quad \lambda_2=\frac{(r+i s)^2}{(r-i s)^2}.\]
Using Theorem \ref{theorem:CMC-connection}, we compute (with $z=x+iy$)
\[\Phi=\frac{\pi (r+is)}{4 rs}\begin{pmatrix}
-1 & ie^{-2\pi i(\frac{x}{r}+\frac{y}{s})}\\
i e^{2\pi i(\frac{x}{r}+\frac{y}{s})}&1\end{pmatrix}dz\quad\mathrm{and}
\quad\Phi^*=\frac{\pi(r-is)}{4rs}\begin{pmatrix}
-1& -i e^{-2\pi i(\frac{x}{r}+\frac{y}{s})}\\
-i e^{2\pi i(\frac{x}{r}+\frac{y}{s})}&1\end{pmatrix}d\bar z.\]
The Willmore energy is
$$\mathcal W(f)=2i\int_{\Sigma}\tr(\Phi\wedge\Phi^*)=
2i\int_{\Sigma}\frac{\pi^2}{4r^2s^2}dz\wedge d\bar z=\frac{\pi^2}{rs}.$$
Consider the gauge
\[g(\xi)=\left(
\begin{array}{cc}
 \frac{i (\xi^2 -1) e^{-i \pi  \left(\frac{x}{r}+\frac{y}{s}\right)}}{4 \xi} & e^{-i \pi  \left(\frac{x}{r}+\frac{y}{s}\right)} \\
 -\frac{\left(\xi+1\right)^2 e^{i \pi  \left(\frac{x}{r}+\frac{y}{s}\right)}}{4 \xi} & \frac{i \left(\xi-1\right) e^{i \pi  \left(\frac{x}{r}+\frac{y}{s}\right)}}{\xi+1} \\
\end{array}
\right)\otimes \mu\]
parametrized by the spectral curve $\xi^2=\lambda$, 
for some $\xi$-independent double-valued function $\mu\colon \C/(r\Z+is\Z)\to\mathbb S^1$.
We obtain
\[D^\xi:=\nabla^{\xi^2}.g=\left(d+\left(
\begin{array}{cc}
 \frac{\pi  (r+i s)}{2 \xi r s} & 0 \\
 0 & -\frac{\pi  (r+i s)}{2 \xi r s} \\
\end{array}
\right)dz+\left(
\begin{array}{cc}
 -\frac{\pi  \xi (r-i s)}{2 r s} & 0 \\
 0 & \frac{\pi  \xi (r-i s)}{2 r s} \\
\end{array}
\right) d\bar z\right).\mu.\]
Let
$\xi_1=1$ and $\xi_2=\frac{r+i s}{r-i s}.$
Then
\[D^{\xi_1}=\left (d+
\pi i\begin{pmatrix}\frac{dx}{r}+\frac{dy}{s}&0\\0&-\frac{dx}{r}-\frac{dy}{s}\end{pmatrix}\right).\mu\]
and
\[D^{\xi_2}=\left(d+\pi i\begin{pmatrix}
-\frac{dx}{r}+\frac{dy}{s}&0\\0&\frac{dx}{r}-\frac{dy}{s}\end{pmatrix}\right).\mu\]
are  both trivial connections  on $\C/(r\Z+i s\Z)$, which are gauge equivalent by the gauge
\[g_1=\begin{pmatrix}e^{-2\pi i\frac{x}{r}}&0\\0&e^{2\pi i\frac{x}{r}}\end{pmatrix}.\]
Since $g_1$ depends only on one variable, there exists a suitable extension $\hat g_1$ of $g_1$ to $B$
such that $\hat g_1^*\omega\wedge \hat g_1^*\omega\wedge \hat g_1^*\omega=0.$
Therefore, we obtain by Proposition \ref{pro:chs}
\begin{align*}
\Theta(D^{\xi_1},g_1)&=\exp\left(\frac{i}{4\pi}\int_{\Sigma}\tr
\begin{pmatrix}-2\pi i \frac{dx}{r}&0\\0&2\pi i \frac{dx}{r}\end{pmatrix}
\wedge \begin{pmatrix}\pi i\frac{dx}{r}+\frac{dy}{s}&0\\0&-\pi i\frac{dx}{r}-\frac{dy}{s}\end{pmatrix}\right)=\exp\left(\frac{i}{4\pi}\int_{\Sigma}\frac{4\pi^2}{rs} dx\wedge dy\right)\\
&=\exp(\pi i)=-1.
\end{align*}
Let $X(\xi)$ be the parallel transport of $X(1)=1$ on the curve $D^{\xi}$ from $\xi_1$ to $\xi_2$ along the unit circle.
Write $D^{\xi}=d+A(\xi)$. Then
\[A'(\xi)=
\begin{pmatrix}
 \frac{-\pi  (r+i s)}{2 \xi^2 r s} & 0 \\
 0 & \frac{\pi  (r+i s)}{2 \xi^2 r s} \end{pmatrix}dz
 +\begin{pmatrix}
 -\frac{\pi  (r-i s)}{2 r s} & 0 \\
 0 & \frac{ \pi (r-i s)}{2 r s} \\
\end{pmatrix} d\bar z\]
\[\tr(A\wedge A')=\frac{2\pi^2 i}{r^2 s^2\xi} dx\wedge dy.\]
We therefore obtain
\[X(\xi_2)=\exp\left(\frac{-i}{4\pi}\int_{\xi_1}^{\xi_2}\int_{\Sigma}\tr(A\wedge A')d\xi\right)
=\exp\left(\frac{\pi}{2rs}\log(\xi_2)\right)
=\exp\left(\frac{\pi}{4rs}\log(\lambda_2)\right).\]
Overall, the holonomy is
\[\Hol(\mathcal D,\gamma)=\exp\left(\frac{\pi}{4 rs}\log(\lambda_2)-\pi i\right).\]
On the other, we have $\log(\lambda_2)=i\tau_2$ and
$\sin(\tau_2)=\Im(\lambda_2)=\Im(r+i s)^4=4 rs(r^2-s^2).$
Hence
\[\exp\left(\frac{i}{4\pi}(\tau_2-\sin(\tau_2))\mathcal W(f)\right)
=\exp\left(\frac{\pi}{4 rs}\log(\lambda_2)-\pi i (r^2-s^2)\right).\]
Theorem \ref{thm:holev} gives
\[\mathcal V(f_{r,s})=2 \pi^2 s^2\]
which can of course also  be verified by elementary means.
\end{Exa}

\section{Meromorphic potentials}\label{sec:dpw}
\subsection{The DPW method}\label{sec:dpwbasic}
We have seen that the construction of CMC surfaces is equivalent to the construction of its associated family of flat connections $\nabla^\lambda.$ Moreover, most geometric information is already contained in the family of gauge classes $[\nabla^\lambda].$
The DPW method \cite{DPW} is a recipe to produce CMC surfaces by taking lifts of $[\nabla^\lambda]$ to the space of flat connections of the form $d+ \eta,$
where $\eta=\eta(z,\lambda)$ is a meromorphic $\Lambda\sl(2,\C)$-valued 1-form, see \cite{HHSch, HHT1,HHTSD}.  More precisely,  $\eta$ --  the DPW potential -- is defined on a compact Riemann surface $\Sigma$ with poles at $p_1,\cdots,p_n$ and is given by
$$\eta(z,\lambda)=\sum_{k=-1}^{\infty}\eta_k(z)\lambda^k,$$
where $\eta_{-1}$ is nilpotent. 
Let $\Sigma^*=\Sigma\setminus\{p_1,\cdots,p_n\}$ and $\PhiDPW$ be a solution on
 of the ODE
$$d\PhiDPW=\PhiDPW\eta$$
on the universal cover $\wt{\Sigma^*}.$

Assume that the following Monodromy Problem is solved:
\begin{equation}
\label{eq:monodromy-problem}
\forall\gamma\in\pi_1(\Sigma),\quad\begin{cases}
\mathcal M(\PhiDPW,\gamma)\in\Lambda SU(2)\\
\mathcal M(\PhiDPW,\gamma)\mid_{\lambda=\lambda_1}=\mathcal M(\PhiDPW,\gamma)\mid_{\lambda=\lambda_2}=\Id,\end{cases}
\end{equation}
where $\lambda_1,\lambda_2\in\S^1$ are the Sym points.
Let $(F,B)\in\Lambda SU(2)\times\Lambda^+_{\R}\mathrm{SL}(2,\C)$ be the loop group Iwasawa decomposition of $\PhiDPW$.
Then $B$ is well-defined on $\Sigma^*$.
From
$dF\,B+F\, dB=FB\eta$
we obtain
\begin{equation}
\label{eq:FinvdF}
F^{-1}dF=B\eta B^{-1}-dB\, B^{-1}
\end{equation}
which is well-defined on $\Sigma^*$.
Write
$$F^{-1}dF=\sum_{k\in\Z}\alpha_k \lambda^k\quad\mathrm{and}\quad B=\sum_{k\in\N} B_k \lambda^k.$$
By Equation \eqref{eq:FinvdF}, $\alpha_k=0$ for $k\leq -2$.
Since $F\in\Lambda SU(2)$, $\alpha_{-k}=-\alpha_k^*$ so $\alpha_k=0$ for $k\geq 2$.
Consider the family of flat connections
$$\nabla^{\lambda}=d+F^{-1}dF=\nabla+\lambda^{-1}\Phi-\lambda\Phi^*$$
with $\nabla=d+\alpha_0$.
Then the Higgs field $\Phi$ is given by
$$\Phi=\alpha_{-1}=B_0\eta_{-1}B_0^{-1}$$
and is a nilpotent 1-form of type $(1,0)$.
Since the Monodromy Problem is solved, $F^{\lambda_1}$ and
$F^{\lambda_2}$ are well defined on $\Sigma^*$ and thus $\nabla^{\lambda_1}$ and
$\nabla^{\lambda_2}$ are trivial.
By Theorem \ref{theorem:CMC-connection}, $\nabla^{\lambda}$ defines a CMC surface $f.$
\subsection{Willmore energy in the DPW method}
The following proposition is proven in \cite{HHT1} in the minimal case. We generalize it here to the CMC case.
\begin{Pro}
\label{Pro:WillmoreDPW}
Let $\Omega\subset\Sigma^*$ be a compact domain. Then the Willmore energy of $f(\Omega)$ is
$$\mathcal W(f\mid_{\Omega})=-2i \int_{\partial\Omega}\tr(\eta_{-1}B_0^{-1}B_1).$$
\end{Pro}
\begin{proof}
From Equation \eqref{eq:FinvdF}, we obtain, since $\alpha_1$ is of type $(0,1)$
$$\Phi^*=-\alpha_1=\frac{\partial}{\partial\lambda}(\overline{\partial} B\, B^{-1})\mid_{\lambda=0}
=\overline{\partial} B_1\,B_0^{-1}-\overline{\partial}B_0\, B_0^{-1}B_1 B_0^{-1}.$$
By Theorem \ref{theorem:CMC-connection}
\begin{align*}
\mathcal W(f\mid_{\Omega})&=2i\int_{\Omega}\tr(\Phi\wedge\Phi^*)=2i\int_{\Omega}\tr\left(B_0\eta_{-1}B_0^{-1}\wedge(
\overline{\partial} B_1\,B_0^{-1}-\overline{\partial}B_0\, B_0^{-1}B_1 B_0^{-1})\right)\\
&=2i\int_{\Omega}\tr\left(\eta_{-1}\wedge(B_0^{-1}\overline{\partial}B_1-B_0^{-1}\overline{\partial}B_0\, B_0^{-1}B_1)\right)=-2i\int_{\Omega}\tr \, d\left(\eta_{-1}B_0^{-1}B_1\right).
\end{align*}
The proposition then follows from Stokes.
\end{proof}
\begin{Def} A local regularizing gauge at $p_j$ is a  meromorphic
gauge $G_j=G_j(z,\lambda)$, defined for $z$ in a neighborhood of $p_j$ (with a pole at $p_j$)
and for $\lambda$ in the unit disk, such that $\eta. G_j$ extends holomorphically at $p_j$.
We say that $p_j$ is an apparent singularity, if such a local regularizing gauge exists.
\end{Def}

The main geometric relevance of the above definition is that the immersion $f$ extends smoothly through any apparent singularity. The following proposition shows that the Willmore energy can be computed in term of regularizing gauges:
\begin{Cor}
Assume that $p_1,\cdots,p_n$ are apparent singularities, with regularizing gauges
$G_1,\cdots,G_n$.
Write
$$G_j=\sum_{k=0}^{\infty} G_{j,k}\lambda^k.$$
Then the Willmore energy of $f$ is given by
$$\mathcal W(f)=4\pi\sum_{j=1}^n\Res_{p_j}\tr(\eta_{-1}G_{j,1}G_{j,0}^{-1}).$$
\end{Cor}
\begin{proof} This is proved in the minimal case in \cite{HHT1} using the Residue Theorem.
The proof carries over to the CMC case through Proposition \ref{Pro:WillmoreDPW}.
\end{proof}

\subsection{Enclosed volume in terms of local regularizing gauges}
We would also like to express the enclosed volume in terms of local regularizing gauges. It appears that this is not as easy as for the Willmore energy and that global aspects are involved.
We shall prove the following result, assuming that the regularizing gauges have a particular form.
\begin{The}
\label{theorem:regularizing-volume}
Assume that the DPW potential $\eta$ has simple poles and the local regularizing gauges all have the form
\begin{equation}
\label{eq:Rj}
G_j(z,\lambda)=N_j(\lambda)\begin{pmatrix}z^{-k_j}&0\\0& z^{k_j}\end{pmatrix},
\end{equation}
where $N_j\in\Lambda_+ \mathrm{SL}(2,\C)$ only depends on $\lambda$, $z$ is a local coordinate in a neighborhood of $p_j$ such that $z(p_j)=0$ and $k_j\in\N^*$.
We also assume that there exists a constant gauge $h\in \mathrm{SL}(2,\C)$ such that
$\eta^{\lambda_2}=\eta^{\lambda_1}. h.$
Let
$$a_j=\big(N_j(\lambda_1)^{-1} h N_j(\lambda_2)\big)_{11}.$$
Then the holonomy of the connection $\mathcal D$ on the Chern-Simons line bundle $\mathcal L$ along $\gamma = [\nabla^{\lambda}]$ is
$$\Hol(\mathcal D,\gamma)=\exp\left(\sum_{j=1}^n  \int_{\lambda_1}^{\lambda_2}\tr\left(\Res_{p_j}\eta\frac{\partial G_j}{\partial\lambda} G_j^{-1}\right) d\lambda
-2 k_j \log(a_j) \right).$$
\end{The}
Theorem \ref{theorem:regularizing-volume} will be proven in the Sections
\ref{section:proof-preliminaries} to \ref{section:inner} below and
we give an example in Section \ref{section:sphere-revisited}.
\begin{Rem} In the case of DPW potentials with simple poles, \eqref{eq:Rj} is the usual form of regularizing gauges.
In particular, Theorem \ref{theorem:regularizing-volume} applies to the Lawson-type CMC surfaces constructed in \cite{HHT2}, see Section \ref{section:Lawson}.
\end{Rem}
\subsubsection{Preliminaries}
\label{section:proof-preliminaries}
\begin{Lem}\
\label{lemma:preliminary}
With the notations introduced above we have
\begin{enumerate}
\item In a neighborhood of $p_j$
$$\eta. N_j=N_j^{-1}\eta N_j=\begin{pmatrix} k_j z^{-1}+O(z^0) & O(z^{-1})\\O(z^{2 k_j})&-k_j z^{-1}+O(z^0)\end{pmatrix} dz.$$

\item The constant matrix $N_j(\lambda_1)^{-1} h N_j(\lambda_2)$ is upper triangular.
\end{enumerate}
\end{Lem}
\begin{Rem}
\label{remark:N_j}
In particular
$$(\Res_{p_j}\eta)N_j=\begin{pmatrix}k_j& \ast\\0&-k_j\end{pmatrix} N_j.$$
Thus the first column of $N_j$ is an eigenvector of $\Res_{p_j}\eta$ with eigenvalue $k_j$, which is useful when constructing the gauge $G_j$.
\end{Rem}

\begin{proof}
Write 
$$\wc\eta=\eta. N_j=\begin{pmatrix}\alpha&\beta\\ \gamma&-\alpha\end{pmatrix} dz$$
with $\alpha,\beta,\gamma$ having at most poles at $z=0$. 
Then
$$\eta. G_j=\wc\eta. \begin{pmatrix}z^{-k_j}&0\\0&z^{k_j}\end{pmatrix}
=\begin{pmatrix}\alpha&\beta z^{2k_j}\\\gamma z^{-2k_j}&-\alpha\end{pmatrix}
+\begin{pmatrix}-k_j&0\\0&k_j\end{pmatrix}\frac{dz}{z}.$$
Since $\eta. G_j$ is holomorphic at $p_j$, the first point follows.
Let
$$\wc h=N_j(\lambda_1)^{-1} h N_j(\lambda_2)=\begin{pmatrix}a_j&b_j\\c_j&d_j\end{pmatrix} \in \mathrm{SL} (2, \C).$$
We have
$$\wc\eta^{\lambda_1}. \wc h=\wc\eta^{\lambda_2}$$
which gives (omitting the indices $j$)
$$\begin{pmatrix}
(ad+bc)\alpha_{\lambda_1}+cd\beta_{\lambda_1}-ab\gamma_{\lambda_1}
& 2 bd\alpha_{\lambda_1}+d^2\beta_{\lambda_1}-b^2\gamma_{\lambda_1}\\
-2 ac\alpha_{\lambda_1}-c^2\beta_{\lambda_1}+a^2\gamma_{\lambda_1}
& -(ad+bc)\alpha_{\lambda_1}-cd\beta_{\lambda_1}+ab\gamma_{\lambda_1}
\end{pmatrix}
=\begin{pmatrix}
\alpha_{\lambda_2}&\beta_{\lambda_2}\\ \gamma_{\lambda_2}&-\alpha_{\lambda_2}
\end{pmatrix}.$$
Taking the residue at $z=0$ of the coefficients $(1,1)$ and $(2,1)$ gives

$$\begin{cases}
(ad+bc) k+ cd \Res\beta_{\lambda_1}=k\\
-2ac k-c^2\Res\beta_{\lambda_1}=0.\end{cases}$$
Elimination of $\Res\beta_{\lambda_1}$ gives $-2ck=0$, so $c=0$.
\end{proof}
\subsubsection{Construction of a global regularizing gauge}
\newcommand{\GG}{{\mathcal G}}
\label{section:global-gauge}
Using cutoff functions, we construct a global, smooth gauge $\GG$ on $\Sigma$ which agrees with the local meromorphic gauge $G_j$ close to $p_j$ and is equal to identity far from $p_1,\cdots,p_n$.
Fixe some small $\varepsilon>0$.
For $\ell=1,2,3$,
let $\psi_\ell:\R\to[0,1]$ be smooth cutoff functions such that
$$
\psi_\ell(r)=\begin{cases}
0\quad\text{if $r\leq \ell\varepsilon$}\\
1\quad\text{if $r\geq (\ell+1)\varepsilon$}.
\end{cases}
$$
Let $H:[0,1]\times\S^1\to \mathrm{SL}(2,\C)$ be a smooth homotopy from
$\begin{pmatrix}e^{-i\theta} & 0\\0& e^{i \theta}\end{pmatrix}$ to $\Id$.
This homotopy exists because $\mathrm{SL}(2,\C)$ is simply connected. Let $\wt N_j:[0,1]\to\Lambda_+ \mathrm{SL}(2,\C)$ be a smooth path from $N_j$ to $\Id$.
Let $D(p_j,r)$ denote the disk $|z|<r$ in the local coordinate $z$ near $p_j$.
We define $\GG_{1}$, $\GG_{2}$, $\GG_{3}$ and $\GG$ in $D(p_j,4\varepsilon)$ 
using the local coordinate $z=r e^{i\theta}$ by
\begin{align*}
&\GG_{1}(r,\theta)=H(\psi_1(r),e^{i k_j\theta})=\begin{cases}
\begin{pmatrix}e^{-i k_j\theta} & 0\\0& e^{i k_j \theta}\end{pmatrix}\quad\text{ if $r\leq \varepsilon$}\\
\Id\quad\text{ if $r\geq 2\varepsilon$}\end{cases}\\
&\GG_{2}(r)=\begin{pmatrix} \rho(r)^{-k_j}&0\\0&\rho(r)^{k_j}\end{pmatrix}
\quad\text{ with }\quad
\rho(r)=r^{1-\psi_2(r)}
=\begin{cases}
r\quad\text{ if $r\leq 2\varepsilon$}\\
1\quad\text{ if $r\geq 3\varepsilon$}\end{cases}\\
&\GG_{3}(r,\lambda)=\wt N_j(\psi_3(r),\lambda)=
\begin{cases} N_j(\lambda)\quad\text{ if $r\leq 3\varepsilon$}\\
\Id\quad\text{ if $r\geq 4\varepsilon$}\end{cases}\\
&\GG=\GG_{3} \GG_{2} \GG_{1},
\end{align*}
which offers a smooth transition from $G_j$ to $\Id$. More precisely,
$$\GG(z,\lambda)=\begin{cases}
\Id\quad\text{ if $r\geq 4\varepsilon$}\\
\GG_{3}(r,\lambda)\quad\text{ if $3\varepsilon\leq r\leq 4\varepsilon$}\\
N_j(\lambda)
\GG_{2}(r)\quad\text{ if $2\varepsilon\leq r\leq 3\varepsilon$}\\
N_j(\lambda)
\begin{pmatrix} r^{-k_j}& 0\\0& r^{k_j}\end{pmatrix}
 \GG_{1}(r,\theta)
\quad\text{ if $\varepsilon\leq r\leq 2\varepsilon$}\\
N_j(\lambda)
\begin{pmatrix} r^{-k_j}& 0\\0& r^{k_j}\end{pmatrix}
\begin{pmatrix}e^{-i k_j\theta} & 0\\0& e^{i k_j\theta}\end{pmatrix}

=G_j (z,\lambda)\quad\text{ if $r\leq \varepsilon$}.\\
\end{cases}$$
Finally, we define $\GG=\Id$ on $\Sigma\setminus\bigcup_{j=1}^n D(p_j,4\varepsilon)$.
The gauge $\GG$ is smooth on $\Sigma\setminus\{p_1,\cdots,p_n\}$.
\begin{Rem}
The definition of the gauge $\GG$ depends on the positive number $\varepsilon$.
Eventually, we let $\varepsilon\to 0$, which allows us to neglect many terms.
\end{Rem}
\subsubsection{Choice of the gauge $g$}
Let
$$\wt\eta=\eta. \GG$$
which is a smooth potential on $\Sigma$, holomorphic in $D(p_j,\varepsilon)$ and
$\Sigma\setminus\bigcup_{i=1}^n D(p_j,4\varepsilon)$ but not holomorphic in the transition annuli.
 As in Theorem \ref{thm:holev},
the holonomy is determined by the integral of the connection 1-form  between $\lambda_1$ and $\lambda_2$ along the unit circle and the Chern-Simons terms caused by the different trivializations 
used at $\lambda_1$ and $\lambda_2$.
Using Proposition \ref{pro:chs}, the holonomy $\Hol(\mathcal D,\gamma)$ is the exponential of
\begin{equation}
\label{eq:logHol}
\frac{1}{4\pi i}\int_{\lambda_1}^{\lambda_2}\int_{\Sigma}\tr\left(\wt\eta\wedge\frac{\partial\wt\eta}{\partial\lambda}\right)d\lambda
+\frac{1}{4\pi i}\int_{\Sigma}\tr\left(\wt\eta^{\lambda_1}\wedge dg\,g^{-1}\right)
-\frac{1}{12\pi i}\int_B\tr\left(\wh{g}^{-1}d\wh{g}\right)^{\wedge 3},
\end{equation}where $g$ is an arbitrary smooth gauge between $\wt\eta^{\lambda_1}$ and
$\wt\eta^{\lambda_2}$ and $\wh g$ is its extension to $B$. Recall that $h$ is a (constant) gauge between $\eta^{\lambda_1}$ and $\eta^{\lambda_2}$.
We take
$$g=\GG_{\lambda_1}^{-1} h \GG_{\lambda_2}$$
so that
$$\wt\eta^{\lambda_1}. g=\eta^{\lambda_1}. \GG_{\lambda_1}\GG_{\lambda_1}^{-1} h \GG_{\lambda_2}=\wt\eta^{\lambda_2}.$$
To see that $g$ is smooth on $\Sigma$, we use Lemma \ref{lemma:preliminary} to write
$$\wc h=N_j(\lambda_1)^{-1}h N_j(\lambda_2) =\begin{pmatrix}a_j&b_j\\0&a_j^{-1}\end{pmatrix}.$$
Then we have in the disk $D(p_j,\varepsilon)$ that
$$g=\begin{pmatrix} a_j& b_jz^{2k_j}\\0& a_j^{-1}\end{pmatrix},$$
which is holomorphic at $p_j$.
\subsubsection{Extension of $g$ to $B$}
Let $B$ be a 3-manifold with $\partial B=\Sigma$.
By compactness, $\Sigma$ has an embedded tubular neighborhood in $B$ which we can identify with $\Sigma\times[0,1]$ via a diffeomorphism.
Let $\chi:\R\to[0,1]$ be a smooth cutoff function such that
$$\chi(t)=\begin{cases}
1\text{ if $t\leq 0$}\\
0\text{ if $t\geq 1$.}\end{cases}$$
We extend $g$ to a smooth map $\wh g:B\to \mathrm{SL}(2,\C)$ as follows :
\begin{itemize}
\item In $N\setminus (\Sigma\times[0,1])$, we take $\wh g=\Id$.
\item In $(\Sigma\setminus\bigcup_{j=1}^n D(p_j,\varepsilon))\times[0,1]$, we define
$$\wh g(\cdot,t)=\wt h(\chi(t))$$
where $\wt h:[0,1]\to \mathrm{SL}(2,\C)$ is a smooth path from $\Id$ to $h$.
\item In the cylinder $D(p_j,3\varepsilon)\times[0,1]$ we define
$$\wh g(z,t)=\GG_{1}(r,\theta)^{-1} \GG_{2}(r)^{-1}\begin{pmatrix}a_j^{\chi(t)}&\chi(t)b_j\\0&a_j^{-\chi(t)}\end{pmatrix}\GG_{2}(r) \GG_{1}(r,\theta)$$
so that $\wh g(z,0)=g(z)$ and $\wh g(z,1)=\Id$. Note that in the cylinder
$D(p_j,\varepsilon)\times[0,1]$, we have
$$\wh g(z,t)=\begin{pmatrix}a_j^{\chi(t)}&\chi(t)b_j z^{2k_j}\\0&a_j^{-\chi(t)}\end{pmatrix}$$
which is smooth at $z=0$.
\item It remains to define $\wh g$ in the domain $3\varepsilon<r<4\varepsilon$, $0<t<1$.
On the boundary of this domain, $\wh g$ is already defined and only depends on $r,t$ (not $\theta$).
Since $\mathrm{SL}(2,\C)$ is simply connected, we can extend this map to the inside of the rectangle $[3\varepsilon,4\varepsilon]\times[0,1]$ as a smooth map depending only on $(r,t)$.
(The rectangle is homeomorphic to the disk so we can extend $\wh g$ as a continuous map, and then we smooth it.)
\end{itemize}
On $\Sigma\setminus\bigcup_{j=1}^n D(p_j,4\varepsilon)$, we have
$\wt\eta=\eta$ which is holomorphic, $g=h$ (a constant) and $\wh g=\wt h(\chi(t))$ only depends on $t$,
so the integrant in each term of \eqref{eq:logHol} vanish.
In $D(p_j,\varepsilon)$, $\wt\eta$, $g$ and $\wh g$ are holomorphic in $z$ so again,
the integrant in each term of \eqref{eq:logHol} vanish.
Therefore, we only have to evaluate the integrals in the transition annuli.
We do this in the next three sections.
\subsubsection{Outer transition annulus}
\label{section:outer}
In this section we fix a singularity $p_j$ and evaluate the integral \eqref{eq:logHol} in the annulus $3\varepsilon\leq r\leq 4\varepsilon$.
\begin{Pro}
\label{prop:outer}
The contribution of the annulus $3\varepsilon\leq r\leq 4\varepsilon$ to
\eqref{eq:logHol} is equal to
$$\int_{\lambda_1}^{\lambda_2}\tr\left(\Res_{p_j}\eta\pd{G_j}{\lambda}G_j^{-1}\right)d\lambda
+O(\varepsilon).$$
\end{Pro}
\begin{proof}
To ease notations, we do not write the dependence of objects on $p_j$, e.g., we write $N=N_j$.
In the annulus $3\varepsilon\leq r\leq 4\varepsilon$, we have $\GG(z,\lambda)=\wt N(\psi_3(r),\lambda)$ which is uniformly bounded
with respect to $\varepsilon$, as is $\pd{\GG}{\lambda}$.
On the other hand,
$$\pd{\GG}{r}=\pd{\wt N}{s}\psi_3'(r)=O(\psi_3'(r)).$$
We write
$$\eta=\left(A(\lambda)+O(z) \right)\frac{dz}{z}\qquad\text{with} \qquad A=\Res_{p_j}\eta.$$
Then

\begin{eqnarray*}
\wt\eta&=&\left(\GG^{-1} A \GG +O(z)\right)\frac{dz}{z}+\GG^{-1}\pd{\GG}{r} dr\\
\pd{\wt\eta}{\lambda}&=&\left(-\GG^{-1}\pd{\GG}{\lambda}\GG^{-1}A \GG+\GG^{-1}\pd{A}{\lambda}\GG+\GG^{-1}A\pd{\GG}{\lambda}+O(z)\right)\frac{dz}{z}\\
&&+
\left(-\GG^{-1}\pd{\GG}{\lambda}\GG^{-1}\pd{\GG}{r}+\GG^{-1}\frac{\partial^2 \GG}{\partial r\partial\lambda}\right) dr.\end{eqnarray*}
Using commutativity of the trace, we obtain
\begin{eqnarray*}
\tr\left(\wt\eta\wedge\pd{\wt\eta}{\lambda}\right)&=&
\tr\left(2A\pd{\GG}{\lambda}\GG^{-1}\pd{\GG}{r}\GG^{-1}-A\pd{\GG}{r}\GG^{-1}\pd{\GG}{\lambda}\GG^{-1}
-A\frac{\partial^2 \GG}{\partial r\partial \lambda}\GG^{-1}\right.\\
&&+\left.\pd{A}{\lambda}\pd{\GG}{r}\GG^{-1}+O(z)\psi_3'(r)\right)
dr\wedge \frac{dz}{z}.\end{eqnarray*}
On the other hand, $dr\wedge\frac{dz}{z}=i\, dr\wedge d\theta$ and
$$\pd{}{\lambda}\left(A\pd{\GG}{r}\GG^{-1}\right)=\pd{A}{\lambda}\pd{\GG}{r}\GG^{-1}+A\frac{\partial^2 \GG}{\partial r\partial\lambda}-A\pd{\GG}{r}\GG^{-1}\pd{\GG}{\lambda}\GG^{-1}$$
$$\pd{}{r}\left(A\pd{\GG}{\lambda}\GG^{-1}\right)=A\frac{\partial^2 \GG}{\partial r\partial\lambda} \GG^{-1}-A\pd{\GG}{\lambda}\GG^{-1}\pd{\GG}{r}\GG^{-1}$$
so
$$\tr\left(\wt\eta\wedge\pd{\wt\eta}{\lambda}\right)=
\left(\tr\pd{}{\lambda}\left(A\pd{\GG}{r}\GG^{-1}\right)
-2\,\tr\pd{}{r}\left(A\pd{\GG}{\lambda}\GG^{-1}\right)+O(z)\psi_3'(r)\right) i\,dr\wedge d\theta.$$
The first term integrates to
\begin{equation}
\label{eq:first-term}\int_{\lambda=\lambda_1}^{\lambda_2}
\int_{r=3\varepsilon}^{4\varepsilon}
\int_{\theta=0}^{2\pi}
\tr\pd{}{\lambda}\left(A\pd{\GG}{r}\GG^{-1}\right)d\lambda
=2\pi i\int_{r=3\varepsilon}^{4\varepsilon}
\tr\left[A\pd{\GG}{r}\GG^{-1}\right]_{\lambda_1}^{\lambda_2}dr.
\end{equation}
The second term integrates to
\begin{eqnarray*}
\int_{\lambda=\lambda_1}^{\lambda_2}
\int_{r=3\varepsilon}^{4\varepsilon}
\int_{\theta=0}^{2\pi}
-2\,\tr\pd{}{r}\left(A\pd{\GG}{\lambda}\GG^{-1}\right)d\lambda
&=&-4\pi i\int_{\lambda_1}^{\lambda_2}\tr\left[A\pd{\GG}{\lambda}\GG^{-1}\right]_{r=3\varepsilon}^{4\varepsilon}\\
&=&4\pi i\int_{\lambda_1}^{\lambda_2}\tr\left(A\pd{N}{\lambda}N^{-1}\right)d\lambda.
\end{eqnarray*}
Finally, the term $O(z)\psi_3'(r)$ integrates to $O(\varepsilon)$.\\

We now consider the second integral in \eqref{eq:logHol}. We have
$$g=\GG_{\lambda_1}^{-1} h \GG_{\lambda_2}$$
$$dg\, g^{-1}=-\GG_{\lambda_1}^{-1} \pd{\GG_{\lambda_1}}{r} dr
+\GG_{\lambda_1}^{-1}h\pd{\GG_{\lambda_2}}{r}\GG_{\lambda_2}^{-1}h^{-1}\GG_{\lambda_1} dr$$
$$\wt \eta^{\lambda_1}=(\GG_{\lambda_1}^{-1} A_{\lambda_1} \GG_{\lambda_1}+O(z))\frac{dz}{z}+\GG_{\lambda_1}^{-1}\pd{\GG_{\lambda_1}}{r}dr.$$
Using commutativity of the trace, we obtain
\begin{eqnarray*}
\tr(\wt\eta^{\lambda_1}\wedge dg\,g^{-1})&=&\tr\left(A_{\lambda_1}\pd{\GG_{\lambda_1}}{r}\GG_{\lambda_1}^{-1}-h^{-1} A_{\lambda_1} h\pd{\GG_{\lambda_2}}{r}\GG_{\lambda_2}^{-1}+O(z)\psi_3'(r)\right)dr\wedge\frac{dz}{z}\\
&=&\tr\left[-A \pd{\GG}{r} \GG^{-1}+O(z)\psi_3'(r)\right]_{\lambda_1}^{\lambda_2} i\, dr\wedge d\theta.\end{eqnarray*}
Therefore, this term integrates to
$$\int_{r=3\varepsilon}^{4\varepsilon}\int_{\theta=0}^{2\pi}
\tr(\wt\eta^{\lambda_1}\wedge dg\,g^{-1})=-2\pi i\int_{r=3\varepsilon}^{4\varepsilon}
\tr\left[A \pd{\GG}{r} \GG^{-1}\right]_{\lambda_1}^{\lambda_2} dr + O(\varepsilon),$$
which cancels with \eqref{eq:first-term}.
Finally, in the annulus $3\varepsilon\leq r\leq 4\varepsilon$, the gauge $\wh g$ only depends on
two variables (namely $r,t$) so $(\wh g^{-1} d\wh g)^{\wedge 3}=0$.
\end{proof}
\subsubsection{Middle transition annulus}
\label{section:middle}
\begin{Pro}
\label{prop:middle}
The contribution of the annulus $2\varepsilon\leq r\leq 3\varepsilon$ to
\eqref{eq:logHol} is $O(\varepsilon)$.
\end{Pro}
\begin{proof}
In the annulus $2\varepsilon\leq r\leq 3\varepsilon$, we have $\GG=N \GG_{2}$.
With the notations of the proof of Lemma \ref{lemma:preliminary}
$$\wt\eta=\wc\eta. \GG_{2}=\begin{pmatrix}\alpha&\beta\rho^{2k}\\\gamma\rho^{-2k}&-\alpha\end{pmatrix}
dz+\begin{pmatrix}-k&0\\0&k\end{pmatrix}\frac{\rho'}{\rho}dr$$
$$\pd{\wt\eta}{\lambda}=\begin{pmatrix}\pd{\alpha}{\lambda}&\pd{\beta}{\lambda}\rho^{2k}\\
\pd{\gamma}{\lambda}\rho^{-2k}&-\pd{\alpha}{\lambda}\end{pmatrix}dz$$
$$\tr\left(\wt\eta\wedge\pd{\wt\eta}{\lambda}\right)=-2k\pd{\alpha}{\lambda}
\frac{\rho'}{\rho} r i e^{i\theta}dr\wedge d\theta.$$
Since $\pd{\alpha}{\lambda}=O(z^0)$ and $\rho\geq r$,
the integral of this term is $O(\varepsilon)$.
We have
$$g=\GG_{2}^{-1} \wc h \GG_{2}=\begin{pmatrix} a & b \rho^{2k}\\0&a^{-1}\end{pmatrix}$$
$$dg\,g^{-1}=\begin{pmatrix} 0 & 2k ab\rho^{2k-1}\\0&0\end{pmatrix}\rho' dr$$
$$\tr(\wt\eta^{\lambda_1}\wedge dg\,g^{-1})=
-2k ab \gamma_{\lambda_1}\frac{\rho'}{\rho} r i e^{i\theta}dr\wedge d\theta.$$
Since $\gamma=O(z^2)$, the integral of this term is $O(\varepsilon^2)$.
As in Section \ref{section:outer}, $\wh g$ only depends on the variables $(r,t)$
so $(\wh g^{-1}d\wh g)^{\wedge 3}=0$.
\end{proof}
\subsubsection{Inner transition annulus}
\label{section:inner}
\begin{Pro}
\label{prop:inner}
The contribution of the annulus $\varepsilon\leq r\leq 2\varepsilon$ to
\eqref{eq:logHol} is $-2 k_j \log(a_j)+O(\varepsilon)$.
\end{Pro}
\begin{proof}
Let $\wf\eta=\eta. N \GG_{2}=\wc\eta.\GG_{2}$. In the annulus $\varepsilon\leq r\leq 2\varepsilon$, we have by Lemma \ref{lemma:preliminary}
$$\wf\eta=\begin{pmatrix} \alpha &\beta r^{2k}\\\gamma r^{-2k}&-\alpha\end{pmatrix} dz
+\begin{pmatrix}-k&0\\0 &k\end{pmatrix}\frac{dr}{r}
=\begin{pmatrix}ki&0\\0&-ki\end{pmatrix} d\theta+O(1) dz$$
$$\wt\eta=\wf\eta. \GG_{1}=\GG_{1}^{-1}\begin{pmatrix}ki&0\\0&-ki\end{pmatrix} \GG_{1} d\theta+O(1) dz+\GG_{1}^{-1}\pd{\GG_{1}}{r}dr+\GG_{1}^{-1}\pd{\GG_{1}}{\theta}d\theta.$$
Since $\GG_{1}$ does not depend on $\lambda$
\begin{equation*}
\begin{split}\pd{\wt\eta}{\lambda}&=O(1) dz\\
\tr\left(\wt\eta\wedge\pd{\wt\eta}{\lambda}\right)&=O(1)d\theta\wedge dz+O(1)\psi_1' (r)dr\wedge dz
\end{split}
\end{equation*}
so this term integrates to $O(\varepsilon)$.
In the annulus $\varepsilon\leq r\leq 2\varepsilon$ we have
$$g=\GG_{1}^{-1}\begin{pmatrix} a& b r^{2k}\\0&a^{-1}\end{pmatrix} \GG_{1}
=\GG_{1}^{-1}(\Delta +O(r^{2k}))\GG_{1}$$
with
$$\Delta=\begin{pmatrix}a&0\\0&\frac{1}{a}\end{pmatrix}.$$
Since $\pd{\GG_{1}}{r}=O(1)\psi'_1$ and $\pd{\GG_{1}}{\theta}=O(1)$
\begin{align*}dg\, g^{-1}=&
\left(
-\GG_{1}^{-1}\pd{\GG_{1}}{r}+\GG_{1}^{-1}\Delta \pd{\GG_{1}}{r}\GG_{1}^{-1}\Delta^{-1} \GG_{1}
+O(r^{2k})\psi'_1+O(r^{2k-1})\right)dr\\
&+\left(
-\GG_{1}^{-1}\pd{\GG_{1}}{\theta}+\GG_{1}^{-1}\Delta \pd{\GG_{1}}{\theta}\GG_{1}^{-1}\Delta^{-1} \GG_{1}
+O(r^{2k})\right)d\theta
\end{align*}
$$\wt\eta^{\lambda_1}=\left(\GG_{1}^{-1}\pd{\GG_{1}}{r}+O(1)\right)dr
+\left(\GG_{1}^{-1}K \GG_{1}
+\GG_{1}^{-1}\pd{\GG_{1}}{\theta}+O(r)\right)d\theta$$
with
$$K=\begin{pmatrix}ki&0\\0&-ki\end{pmatrix}.$$
Hence
\begin{align*}
\tr(\wt\eta^{\lambda_1}\wedge dg\,g^{-1})
=&\Big(
-\GG_{1}^{-1}\pd{\GG_{1}}{r}\GG_{1}^{-1}\pd{\GG_{1}}{\theta}+\GG_{1}^{-1}\pd{\GG_{1}}{r}
\GG_{1}^{-1}\Delta\pd{\GG_{1}}{\theta}\GG_{1}^{-1}\Delta^{-1}\GG_{1}
+ \GG_{1}^{-1}K\pd{\GG_{1}}{r}\\
&- \GG_{1}^{-1}K\Delta\pd{\GG_{1}}{r}\GG_{1}^{-1}\Delta^{-1}\GG_{1} 
+\GG_{1}^{-1}\pd{\GG_{1}}{\theta}\GG_{1}^{-1}\pd{\GG_{1}}{r}\\
&-\GG_{1}^{-1}\pd{\GG_{1}}{\theta}\GG_{1}^{-1}\Delta\pd{\GG_{1}}{r}\GG_{1}^{-1}\Delta^{-1}\GG_{1}
+O(r)\psi'_1+O(r^{2k-1})\Big)dr\wedge d\theta.
\end{align*}
Let
$$M_r=\pd{\GG_{1}}{r}\GG_{1}^{-1}\quad\text{ and }\quad
M_{\theta}=\pd{\GG_{1}}{\theta}\GG_{1}^{-1}.$$
Using commutativity of the trace and $\Delta^{-1} K\Delta=K$, we obtain
$$\tr(\wt\eta^{\lambda_1}\wedge dg\,g^{-1})
=\tr\left(M_r\Delta M_{\theta}\Delta^{-1}-M_{\theta}\Delta M_r\Delta^{-1}
+O(r)\psi_1'+O(r^{2k-1})\right)dr\wedge d\theta.$$
Now for any matrices $A,B$, an elementary computation shows that
$$\tr\left(A\Delta B\Delta^{-1}-B\Delta A\Delta^{-1}\right)=
 (a^2-a^{-2})[A,B]_{22}.$$
 Using Lemma \ref{lemma:M1M2} below, we obtain
\begin{equation}
\label{eq:second-term}
\int_{r=\varepsilon}^{2\varepsilon}\int_{\theta=0}^{2\pi}
\tr(\wt\eta^{\lambda_1}\wedge dg\,g^{-1})
=-2\pi k i(a^2-a^{-2})+O(\varepsilon).
\end{equation}
\begin{Lem} It holds
\label{lemma:M1M2}
$$\int_{r=\varepsilon}^{2\varepsilon}\int_{\theta=0}^{2\pi}
[M_r,M_{\theta}]dr \wedge d\theta=\begin{pmatrix}2\pi k i&0\\0&-2\pi k i\end{pmatrix}.$$
\end{Lem}
\begin{proof}
We have
\begin{align*}
d\left(\pd{\GG_{1}}{\theta}\GG_{1}^{-1}d\theta+\pd{\GG_{1}}{r}\GG_{1}^{-1}dr\right)
&=\left(\frac{\partial^2 \GG_{1}}{\partial r\partial\theta}\GG_{1}^{-1}
-\pd{\GG_{1}}{\theta}\GG_{1}^{-1}\pd{\GG_{1}}{r}\GG_{1}^{-1}
-\frac{\partial^2 \GG_{1}}{\partial \theta\partial r}\GG_{1}^{-1}
+\pd{\GG_{1}}{r}\GG_{1}^{-1}\pd{\GG_{1}}{\theta}\GG_{1}^{-1}\right)dr\wedge d\theta\\
&=[M_r,M_{\theta}]dr\wedge d\theta.\end{align*}
By Stokes formula
$$\int_{r=\varepsilon}^{2\varepsilon}[M_r,M_{\theta}]dr\wedge d\theta
=\int_{\theta=0}^{2\pi}\left[\pd{\GG_{1}}{\theta}\GG_{1}^{-1}\right]_{r=\varepsilon}^{2\varepsilon}d\theta\\
=-\int_0^{2\pi}\begin{pmatrix}-k i&0\\0&k i\end{pmatrix}d\theta\\
=\begin{pmatrix}2\pi k i&0\\0&-2\pi k i\end{pmatrix}.$$
\end{proof}
Returning to the proof of Proposition \ref{prop:inner}, we now consider the third term in \eqref{eq:logHol}.
We have in the annulus $\varepsilon\leq r\leq 2\varepsilon$
\begin{align*}
\wh g=&\GG_{1}(r,\theta)^{-1}\begin{pmatrix}
a^{\chi(t)}&\chi(t) b r^{2k}\\0& a^{-\chi(t)}\end{pmatrix} \GG_{1}(r,\theta)
=\GG_{1}(r,\theta)^{-1}(\Delta^{\chi(t)}+O(r^2))\GG_{1}(r,\theta)\\
\wh g^{-1}d\wh g=&
\left(-\GG_{1}^{-1}\Delta^{-\chi}\pd{\GG_{1}}{r}\GG_{1}^{-1}\Delta^{\chi} \GG_{1}
+\GG_{1}^{-1}\pd{\GG_{1}}{r}+O(r^{2k})\psi_1'+O(r^{2k-1})\right)dr\\
&+\left(-\GG_{1}^{-1}\Delta^{-\chi}\pd{\GG_{1}}{\theta}\GG_{1}^{-1}\Delta^{\chi} \GG_{1}
+\GG_{1}^{-1}\pd{\GG_{1}}{\theta}+O(r^{2k})\right)d\theta\\
&+\left(\GG_{1}^{-1}L \GG_{1}+O(r^{2k})\right)\chi'dt
\end{align*}
with
$$L=\begin{pmatrix} \log(a)&0\\0&-\log(a)\end{pmatrix}.$$
Using
$$\tr\left((A\,dx+ B\,dy +C\, dz)^{\wedge 3}\right)=3\,\tr(ABC-BAC) dx\wedge dy \wedge dz$$
we obtain
\begin{eqnarray*}
\lefteqn{\tr\left(\wh g^{-1}d\wh g)^{\wedge 3}\right)
=3\,\tr\Big(
\GG_{1}^{-1}\Delta^{-\chi}\pd{\GG_{1}}{r}\GG_{1}^{-1}\pd{\GG_{1}}{\theta}\GG_{1}^{-1}\Delta^{\chi}L \GG_{1}
-\GG_{1}^{-1}\Delta^{-\chi}\pd{\GG_{1}}{r}\GG_{1}^{-1}\Delta^{\chi}\pd{\GG_{1}}{\theta}\GG_{1}^{-1}L \GG_{1}
}\\
&&-\GG_{1}^{-1}\pd{\GG_{1}}{r}\GG_{1}^{-1}\Delta^{-\chi}\pd{\GG_{1}}{\theta}\GG_{1}^{-1}\Delta^{\chi}L \GG_{1}
+\GG_{1}^{-1}\pd{\GG_{1}}{r}\GG_{1}^{-1}\pd{\GG_{1}}{\theta}\GG_{1}^{-1} L \GG_{1}\\
&&+\mbox{ (same exchanging $r$ and $\theta$) } + O(r^{2k})\psi'_1+O(r^2k-1)\Big)
dr\wedge d\theta\wedge \chi' dt\\
&=&
3\,\tr\left(\left[M_r-\Delta^{-\chi}M_r\Delta^{\chi},M_{\theta}\right]L
+\left[M_r,M_{\theta}-\Delta^{-\chi}M_{\theta}\Delta^{\chi}\right]L +O(r^{2k})\psi'_1+O(r^{2k-1})\right)
dr\wedge d\theta\wedge\chi' dt.
\end{eqnarray*}
An elementary computation shows that for any matrices $A,B$
$$\tr\left([A-\Delta^{-\chi}A\Delta^{\chi},B]L+[A,B-\Delta^{-\chi}B\Delta^{\chi}]L\right)
=2(a^{\chi}-a^{-\chi})^2\log(a)[A,B]_{22}.$$
Hence using Lemma \ref{lemma:M1M2}
\begin{eqnarray*}
\int_{t=1}^0\int_{r=\varepsilon}^{2\varepsilon}\int_{\theta=0}^{2\pi}
\tr\left(\wh g^{-1}d\wh g)^{\wedge 3}\right)
&=&6\int_{t=1}^0\left(a^{\chi(t)}-a^{-\chi(t)}\right)^2\log(a)\chi'(t)dt
\int_{r=\varepsilon}^{2\varepsilon}\int_{\theta=0}^{2\pi}
[M_r,M_{\theta}]_{22}dr\,d\theta +O(\varepsilon)\\
&=&-12\pi k i\int_{s=0}^1 (a^{2s}+a^{-2s}-2)\log(a) ds+O(\varepsilon)\\
&=&-6\pi k i(a^2-a^{-2})+24\pi  k i \log(a) +O(\varepsilon).
\end{eqnarray*}
We integrate from $t=1$ to $t=0$ because $\Sigma$ is oriented as
a boundary of $B$.
The first term cancels with \eqref{eq:second-term} when computing \eqref{eq:logHol}.
This concludes the proof of Proposition \ref{prop:inner}.
Theorem \ref{theorem:regularizing-volume} follows by letting $\varepsilon\to 0$.
\end{proof}

Sometimes, positive regularizing gauges can be quite complicated. The following remark enables us to compute the volume using simpler gauges which are not defined in the whole unit $\lambda$-disk.
\begin{Rem}
\label{remark:fake-gauge}
Assume that for each apparent singularity $p_j$, we are given another local gauge $G'_j$ such that $\eta. G'_j$ extends holomorphically to $p_j$ and which has the same form
$$G'_j(z,\lambda)=N'_j(\lambda)\begin{pmatrix}z^{-k_j}&0\\0&z^{k_j}\end{pmatrix}$$
but
$N'_j(\lambda)$ is only supposed to be defined for $\lambda$ in a neighborhood of the
arc $[\lambda_1,\lambda_2]$ on the unit circle.
By Lemma \ref{lemma:preliminary}, the first column of $N'_j$ is an eigenvector  of
$\Res_{p_j}\eta$ for the eigenvalue $k_j$, so $N_j^{-1}N'_j$ is upper triangular.
Hence $G_j^{-1} G'_j$ extends holomorphically to $p_j$.
Let $\GG'$ be the global gauge on $\Sigma$ constructed as in Section \ref{section:global-gauge} using the local gauges $G'_j$. Then $\GG^{-1} \GG'$ extends holomorphically at $p_j$
and is a smooth gauge between $\eta.\GG$ and $\eta.\GG'$ (for $\lambda$ on the arc $[\lambda_1,\lambda_2]$). So we can use $\GG'$ to compute the holonomy
$\Hol(\mathcal D,\gamma)$, and Theorem \ref{theorem:regularizing-volume} also holds for the gauge $\GG'$ (since we did not use in the proof that $G_j$ is positive).
\end{Rem}
\subsubsection{Example: spheres revisited}
\label{section:sphere-revisited}
We consider the following Fuchsian potential for CMC spheres :
$$\eta=\begin{pmatrix} 0&\lambda^{-1}\\\lambda &0\end{pmatrix}\frac{dz}{z}.$$
This constructs double covers of round spheres, branched over $0$ and $\infty$.
Local regularizing gauges at $0$ and $\infty$ are
$$G_0=\begin{pmatrix}1&0\\ \lambda&1\end{pmatrix}\begin{pmatrix}z^{-1}&0\\0&z\end{pmatrix}$$
$$G_{\infty}=\begin{pmatrix}1&0\\-\lambda& 1\end{pmatrix}\begin{pmatrix}w^{-1}&0\\0&w\end{pmatrix}$$
with $w=\frac{1}{z}$ as local coordinate in a neighborhood of $\infty$.
With Sym-points $\lambda_1=e^{-i\alpha}$, $\lambda_2=e^{i\alpha}$, a gauge between
$\eta^{\lambda_1}$ and $\eta^{\lambda_2}$ is
$$h=\begin{pmatrix} e^{i\alpha}&0\\0&e^{-i\alpha}\end{pmatrix}.$$
We obtain
$$\tr\left(\Res_0\eta \pd{G_0}{\lambda} G_0^{-1}\right)=
\tr\left(\Res_{\infty}\eta \pd{G_{\infty}}{\lambda} G_{\infty}^{-1}\right)=\lambda^{-1}$$
$$a_0=a_{\infty}=e^{i\alpha}$$
which gives
$$\Hol(\gamma)=\exp\left(2\int_{e^{-i\alpha}}^{e^{i \alpha}}\frac{d\lambda}{\lambda}-4\log(i\alpha)\right)=1$$
as expected.
\section{Enclosed volume for Lawson-type CMC surfaces}
\label{section:Lawson}
Lawson \cite{L} has constructed a family of embedded minimal surfaces $\xi_{1,g}$ in $\S^3$ parameterized by their genus $g\in\N^*$.
They admit an orientation reversing symmetry which interchanges inside and outside, so their enclosed volume  is always $\pi^2$.
In \cite{HHT2}, we have constructed, for large genus $g$, deformations of Lawson minimal surfaces
$\xi_{1,g}$ which are closed embedded CMC surfaces in $\S^3$ (not to be confused with Lawson CMC surfaces in $\R^3$  obtained via Lawson correspondence).
We computed the order 2 expansion of their Willmore energy in terms of $1/g$.
\subsection{Fuchsian potential for Lawson-type CMC surfaces in $\S^3$}\label{sec:LCMC}
In this section, we recall the construction in \cite{HHT2}.
Fix $\varphi\in(0,\pi/2)$.
We consider the 4-punctured sphere $\C P^1\setminus\{p_1,p_2,p_3,p_4\}$
with
$$p_1=e^{i\varphi},\quad p_2=-e^{-i\varphi},\quad p_3=-e^{i\varphi},\quad p_4=e^{-i\varphi}.$$
We consider a meromorphic potential $\eta$ of the form
\begin{equation}\label{eq:eta}\eta= t\sum_{k=1}^4 A_k(\lambda)\frac{dz}{z-p_k}\end{equation}
with
\begin{equation}
\label{eq:etaAk}
\begin{split}
A_1&=\begin{pmatrix}x_1 & x_2+i x_3\\ x_2-i x_3& -x_1\end{pmatrix}\qquad\qquad
A_2=\begin{pmatrix}-x_1 & -x_2+i x_3\\ -x_2-i x_3& x_1\end{pmatrix}\\
A_3&=\begin{pmatrix}x_1 & -x_2-i x_3\\ -x_2+i x_3& -x_1\end{pmatrix}\qquad\qquad
A_4=\begin{pmatrix}-x_1 & x_2-i x_3\\ x_2+i x_3& x_1\end{pmatrix}.
\end{split}
\end{equation}
Here $t>0$ is the main parameter of the construction and $x_1,x_2,x_3$ are meromorphic functions of $\lambda$ in the unit disk with simple poles at $\lambda=0$ with residues
$$\Res_{\lambda=0}(x_1,x_2,x_3)=\tfrac{1}{2}\left(i,-\sin(\varphi),-\cos(\varphi)\right).$$
(These are chosen so that $\Res_{\lambda=0}\eta$ is nilpotent.)
Let $\PhiDPW$ be the solution of
$d\PhiDPW=\PhiDPW\eta$ with initial condition $\PhiDPW(z=0)=\Id$ and
$M_k$ the monodromy of $\PhiDPW$ around $p_k$.
We have solved \cite{HHT2} the following Monodromy Problem
\begin{equation}
\label{eq:monodromy-problem2}\begin{cases}
\exists U\in \mathrm{SL}(2,\C),\;\forall k,\; U^{-1}M_k U\in\Lambda SU(2)\\
\exists s>0,\;\forall k,\; M_k\text{ has constant eigenvalues } e^{\pm 2\pi i s}\\
\exists \theta\in\R,\;\forall k,\; M_k(\lambda=e^{\pm i\theta})\text{ is diagonal}
\end{cases}\end{equation}
using the Implicit Function Theorem at $t=0$, which determines the functions $x_1,x_2,x_3$, and the real numbers $\theta,s$ as functions of $t$ for $t>0$ small enough. In particular, at $t=0$, we have
\begin{equation}
\label{eq:central}
\begin{cases}
x_1&=\tfrac{i}{2}(\lambda^{-1}-\lambda)\\
x_2&=\tfrac{-1}{2}\sin(\varphi)(\lambda^{-1}+\lambda)\\
x_3&=\tfrac{-1}{2}\cos(\varphi)(\lambda^{-1}+\lambda)\\
\theta&=\frac{\pi}{2}
\end{cases}.
\end{equation}
Moreover, $s(t)\sim t$.
Then for $g\in\N$ large enough, we consider the compact Riemann surface $\Sigma_g$
of genus $g$ defined by the algebraic equation
$$y^{g+1}=\frac{(z-p_1)(z-p_3)}{(z-p_2)(z-p_4)},$$
and  the $(g+1)$-fold covering $\pi_g\colon\Sigma_g\to\C P^1$, $(y,z)\mapsto z$ which is totally branched over
$p_1,\cdots,p_4$.
We take $t$ such that
$$s(t)=\frac{1}{2g+2}$$
and let $\wt\eta_g=\pi_g^*\eta$ be the lift of $\eta$ to $\Sigma_g$.
We take the Sym points
$\lambda_1=e^{-i\theta(t)}$ and $\lambda_2=e^{i\theta(t)}.$
The Monodromy Problem \eqref{eq:monodromy-problem} is solved and
moreover, $\wt{p}_k=\pi^{-1}(p_k)$ are apparent singularities. We obtain,
for each $\varphi\in(0,\pi/2)$ and $g$ large enough,
a CMC immersion $f_g:\Sigma_g\to\S^3$.
For large $g$, the image $f_g(\Sigma_g)$ looks like the desingularization of two great spheres meeting with an angle $2\varphi$.
In the case $\varphi=\pi/4$, we obtain Lawson's minimal surface $\xi_{1,g}$.

An important aspect of the construction is that the $t$-derivatives of the parameters
$x_1,x_2,x_3,\theta$ and $s$ at $t=0$ can be computed, up to any order, in terms of certain iterated integrals.
In particular, the derivatives of order up to $2$ are explicit, elementary functions of $\theta$.

\begin{Rem}\label{etarescaled}
By standard Fuchsian theory, the monodromy $M_k$ is conjugate to
$\exp(2\pi i t A_k)$. Hence, by the second item of the Monodromy Problem
\eqref{eq:monodromy-problem2}, $t A_k$ has eigenvalues $\pm s$, so
$$\det(tA_k)=-s^2=-t^2(x_1^2+x_2^2+x_3^2).$$
It will be convenient to scale $x_k$ by $t/s$, so that we have
$x_1(t)^2+x_2(t)^2+x_3(t)^2=1$ after scaling.
This scales the matrices $A_k$ by $t/s$ as well, and by abuse of notation the new potential is given by
$$\eta=s\sum_{k=1}^4 A_k\frac{dz}{z-p_k}.$$
\end{Rem}

\begin{Pro}
\label{prop:diagonal}
For all $t$, it holds
$$(x_1,x_2,x_3)\mid_{\lambda=\lambda_1}=(-1,0,0)\quad\text{ and }\quad
(x_1,x_2,x_3)\mid_{\lambda=\lambda_2}=(1,0,0).$$
\end{Pro}
\begin{proof}
Fix $t$ and $\lambda\in\{\lambda_1,\lambda_2\}$. By the second item of the Monodromy Problem \eqref{eq:monodromy-problem2}, the monodromy representation of $\PhiDPW$ is diagonal.
Hence the rows of $\PhiDPW$ define two invariant line bundles, and
the connection $\nabla=d+\eta$ is  reducible (and unitary).
By Lemma 11 in \cite{Fuchsian}, $\eta$ is conjugate (by a constant matrix) to a diagonal potential. In other words, the matrices $A_1,\cdots A_4$ are simultaneously diagonalizable and commute. But
$$[A_1,A_2]=\begin{pmatrix}-4i x_2 x_3& 4i x_1 x_3\\4i x_1 x_3&4 i x_2 x_3\end{pmatrix}\quad\text{ and }\quad
[A_1,A_4]=\begin{pmatrix} 4 i x_2 x_3& 4x_1 x_2\\-4x_1 x_2&-4i x_2 x_3\end{pmatrix}$$
so two amongst the three $\{x_1$, $x_2$, $x_3\}$ are zero. Since $x_1^2+x_2^2+x_3^2=1$, we have
$$(x_1,x_2,x_3)\in\{(\pm 1, 0,0),(0,\pm 1,0),(0,0,\pm 1)\}.$$
Now $(x_1,x_2,x_3)$ depends continuously on $t$ and takes values in a discrete set so is constant. At $t=0$, we have $\lambda_1=-i$ and $\lambda_2=i$. We conclude using the values of $x_1,x_2,x_3$ at $t=0$ given in \eqref{eq:central}.
\end{proof}
\subsection{Regularizing gauges}
As in \cite{HHT2}, to compute the Willmore energy or the enclosed volume, we need to work on a double cover of $\Sigma_g$ so we consider instead the equation
$$y^{2g+2}=\frac{(z-p_1)(z-p_3)}{(z-p_2)(z-p_4)}.$$
The lifted potential $\wt\eta_g$ has a simple pole at $\wt p_1,\cdots,\wt p_4$ with residue
$\Res_{\wt p_k}\wt\eta_g=(2g+2)sA_k= A_k.$
Let $w$ be a local coordinate in a neighborhood of $\wt p_1$ such that $w(p_1)=0$.
Then,
$$\wt\eta_g=A_k\frac{dw}{w}+O(w^{2g+1})dw.$$
The regularizing gauge at $p_1$ used in \cite{HHT2} is
$$G_1(z,\lambda)=\begin{pmatrix}1&0\\ \kappa_1&1\end{pmatrix}
\begin{pmatrix}w^{-1}&0\\0&w\end{pmatrix}
\quad\text{ with }\quad
\kappa_1=\frac{1-x_1}{x_2+i x_3}$$
However, by Proposition \ref{prop:diagonal}, $\kappa_1(\lambda_1)=\infty$. This is not a problem for computing the immersion (see Remark 21 in \cite{HHT2}) or the Willmore energy, but we cannot compute the holonomy using this gauge since it is not defined at $\lambda= \lambda_1$.

This problem can be fixed by considering the following gauge instead:
\begin{equation}
\label{eq:true-gauge}
G_1(z,\lambda)=\begin{pmatrix}\frac{x_2+i x_3}{\Delta}&-1\\\frac{1-x_1}{\Delta}& e^{i\varphi}\end{pmatrix}
\begin{pmatrix}w^{-1}&0\\0&w\end{pmatrix}
\quad\text{ with }\quad
\Delta=1-x_1+e^{i\varphi}(x_2+i x_3).
\end{equation}
\begin{Pro}\label{prop10}
For $t$ small enough, $G_1$ is a regularizing gauge at $p_1$ which is holomorphic on the closed unit $\lambda$-disk.
\end{Pro}
\begin{proof}
First of all, $G_1$ extends holomorphically at $\lambda=0$ because $x_1,x_2,x_3$ and
$\Delta$ have simple poles there.
By Proposition \ref{prop:diagonal}, $\Delta$ has a zero at $\lambda= \lambda_2$.
At $t=0$, we have
$$\Delta=1-\tfrac{i}{2}(\lambda^{-1}-\lambda)-\tfrac{1}{2}e^{i\varphi}(\sin(\varphi)+i\cos(\varphi))(\lambda^{-1}+\lambda)=1-i\lambda^{-1}.$$
Hence for small $t$, the zero at $\lambda= \lambda_2$ is simple and $\Delta$ has no other zero in the closed unit disk.
So $G_1$ extends holomorphically to $\lambda= \lambda_2$ and is holomorphic in the closed unit disk.
A computation, using $x_1^2+x_2^2+x_3^2=1$, gives
$$\wt\eta_g. G_1=\begin{pmatrix}0 & 2i e^{i\varphi}\left(i x_1+x_2\sin(\varphi)+x_3\cos(\varphi)\right)w\\ 0 & 0\end{pmatrix}dw+O(w^{2g})dw$$
which is holomorphic at $\wt p_1$.
\end{proof}
To compute the holonomy of $\mathcal D$, it will be easier to use the following regularizing gauge:
\begin{equation}
\label{eq:fake-gauge}
G_1=\frac{1}{\sqrt{\Delta}}\begin{pmatrix}x_2+i x_3& x_1-1\\1-x_1&x_2-i x_3\end{pmatrix} \begin{pmatrix}z^{-1}&0\\0&z\end{pmatrix}
\quad\text{ with }\quad
\Delta=x_2^2+x_3^2+1-2x_1+x_1^2=2(1-x_1).
\end{equation}
Note that this gauge is not holomorphic in the unit disk nor well-defined, since 
$\Delta$ has a simple pole at $\lambda=0$.
At $\lambda= \lambda_2$, $1-x_1$ has a double zero, so $\sqrt{\Delta}$ has a simple zero and thus
$G_1$ extends holomorphically to $\lambda= \lambda_2$.
At the central value $t=0$, we have
$\Delta=i\lambda^{-1}(\lambda-i)^2$
so for $t$ small enough, $\Delta$ has no other zero.
So $G_1$ is well defined on the closed arc from $\lambda_1$ to $\lambda_2$ and can be used to compute the holonomy of $\mathcal D$ on $\mathcal L$ by Remark \ref{remark:fake-gauge}.
\begin{Rem}
Note that all the gauges we considered in this section have the form \eqref{eq:Rj}, and
the first column is always collinear to $(x_2+i x_3,1-x_1)$ which is an eigenvector of
$A_1$, as predicted by Remark \ref{remark:N_j}.
\end{Rem}
\subsection{Holonomy}
 We return to the problem of computing the holonomy $\mathcal D$ of the Chern-Simons line bundle $\mathcal L$ over $\mathcal M$ along the curve $[\nabla^\lambda]$ for $\lambda \in S^1.$
\begin{The}\label{The:LeV} The holonomy of $\mathcal D$ along $\gamma= [\nabla^\lambda]$ of the Lawson CMC family is given by
\label{prop:holonomy-fake-gauge}
\begin{equation}\label{eq:LEV}\log[\Hol(\mathcal D,\gamma)]=4\log\left[\frac{x_2'(\lambda_2)-i x_3'(\lambda_2)}{x_2'(\lambda_2)+i x_3'(\lambda_2)}\right]
+4 i\int_{\lambda_1}^{\lambda_2}\left(x_2 x'_3-x_3 x'_2\right)\frac{d\lambda}{1-x_1}\end{equation}
where $()'$ denotes the derivative with respect to $\lambda$.
\end{The}
\begin{proof}
We use Theorem \ref{theorem:regularizing-volume} with the gauge defined in \eqref{eq:fake-gauge}.
By Proposition \ref{prop:diagonal} we have $\eta^{\lambda_2}=-\eta^{\lambda_1},$
so we may take 
$$h=\begin{pmatrix} 0& -1\\1 &0\end{pmatrix}$$
as the (constant) gauge between $\eta^{\lambda_1}$ and $\eta^{\lambda_2}$.
Let
$$N_1=\frac{1}{\sqrt{\Delta}}\begin{pmatrix}x_2+i x_3& x_1-1\\1-x_1&x_2-i x_3\end{pmatrix}.$$
Then
$$N_1(\lambda_1)=\begin{pmatrix}0&-1\\1&0\end{pmatrix}=h\quad
\mathrm{and}
\quad N_1(\lambda_2)=\begin{pmatrix}\ell& 0\\ 0&\ell^{-1}\end{pmatrix}$$
with
$$\ell=\lim_{\lambda\to\lambda_2}\frac{x_2+i x_3}{\sqrt{2(1-x_1)}}
=\lim_{\lambda\to\lambda_2}\frac{x_2+i x_3}{\sqrt{1-x_1^2}}
=\lim_{\lambda\to\lambda_2}\sqrt{\frac{x_2+i x_3}{x_2-i x_3}}.$$
This gives $a_1=\ell$ and
$$-2\log a_1=-2\log\ell=\log\left[\lim_{\lambda\to\lambda_2}\frac{x_2-i x_3}{x_2+i x_3}\right]
=\log \left[\frac{x_2'(\lambda_2)-i x_3'(\lambda_2)}{x_2'(\lambda_2)+i x_3'(\lambda_2)}\right].$$
The computation of $\tr(A_1 N'_1 N_1^{-1})$ is quite simple:
$$N_1^{-1} A_1=\frac{1}{\sqrt{\Delta}}\begin{pmatrix} x_2-i x_3&1-x_1\\1-x_1& -x_2-i x_3\end{pmatrix}$$
$$N'_1 N_1^{-1} A_1=\frac{1}{\Delta}\begin{pmatrix}x'_2+i x'_3&x'_1\\-x'_1&x'_2-i x'_3\end{pmatrix}\begin{pmatrix} x_2-i x_3&1-x_1\\1-x_1& -x_2-i x_3\end{pmatrix}
-\frac{\Delta'}{2\Delta} A_1$$
and since $\tr(A_1)=0$ we obtain
$$\tr(N'_1 N_1^{-1} A_1)=\frac{2i x_2 x'_3-2i x_3 x'_2}{\Delta}.$$
So the contribution of $\wt p_1$ to the holonomy is
$$\log\left[\frac{x_2'(\lambda_2)-i x_3'(\lambda_2)}{x_2'(\lambda_2)+i x_3'(\lambda_2)}\right]
+ i\int_{\lambda_1}^{\lambda_2}\left(x_2 x'_3-x_3 x'_2\right)\frac{d\lambda}{1-x_1}.$$
By symmetry, the contribution of the three other poles $\wt p_2$, $\wt p_3$ and $\wt p_4$ is the same. More precisely, the symmetries are encoded into $A_k=C_k A_1 C_k^{-1}$ with
\begin{equation}\label{eq:Cs}C_2=\begin{pmatrix}0&-1\\1&0\end{pmatrix},\quad
C_3=\begin{pmatrix}i&0\\0&-i\end{pmatrix}\quad\text{ and }\quad
C_4=\begin{pmatrix}0&i\\i&0\end{pmatrix}.\end{equation}
We take $N_k=C_k N_1$. Then for $k=2,3,4$ it holds
$\tr(A_k N'_k N_k)=\tr(A_1 N'_1 N_1)$
and
$$N_k(\lambda_1)^{-1} h N_k(\lambda_2)=N_1(\lambda_1)^{-1} C_k^{-1} h C_k N_1(\lambda_2).$$
Since
$$C_2^{-1}h C_2=h,\quad C_3^{-1} h C_3=-h\quad\text{ and }\quad
C_4^{-1} h C_4=-h$$
we obtain $a_2=a_1$ and $a_4=a_3=-a_1$, hence modulo $2\pi i$ all terms are equal.
\end{proof}

\subsection{Second order expansion of the enclosed volume}
In this section, we compute the second order expansion of the enclosed volume $\mathcal V(f_g)$ with respect to $s=\frac{1}{2g+2}$.
Since we are working on a double cover when computing the holonomy, we have to double the Willmore energy and the volume. Hence by Theorem \ref{thm:holev}
$$\log\Hol(\mathcal D,\gamma)=\frac{i}{4\pi}(2\theta-\sin(2\theta))2\mathcal W-\frac{i}{\pi}2\mathcal V$$
which gives
\begin{equation}
\label{eq:volume}\mathcal V=\frac{(2\theta-\sin(2\theta))}{4}\mathcal W-\frac{\pi}{2i}\log\Hol(\mathcal D,\gamma)\;\mod \pi^2.
\end{equation}
The following quantities will play an important role in our computations:
\begin{equation*}
\begin{split}\WWW&=\cos^2(\varphi)\log(\cos(\varphi))+\sin^2(\varphi)\log(\sin(\varphi))\\
\TTT&=\sin(2\varphi)\log(\tan(\varphi)).
\end{split}
\end{equation*}
We already computed the second order expansion of $\mathcal W$ and $\theta$
in \cite{HHT2}:
$$\mathcal W(f_g)=8\pi+16 \pi \WWW\, s+O(s^3).$$
$$\theta(s)=\frac{\pi}{2}+2 \TTT\, s+O(s^3).$$
\begin{The}
\label{theorem:volume-order2}
The second order expansion of the enclosed volume of Lawson CMC surfaces is
$$\mathcal V(f_g)=2\pi^2-4\pi\varphi 
+ 16\pi \TTT s
+ 16\pi \WWW\TTT s^2+O(s^3)
\quad\text{ with }\quad s=\frac{1}{2g+2}.
$$
\end{The}
We prove Theorem \ref{theorem:volume-order2} in the next three sections using Proposition \ref{prop:holonomy-fake-gauge}.
We will in fact compute the order 2 expansion with respect to the parameter $t$.
But since $s(t)=t+O(t^3)$, this gives the order 2 expansion with respect to $s$ as well.

\subsubsection{Central value}
By Equations \eqref{eq:central}, we have at $t=0$
$$\frac{x'_2-i x'_3}{x'_2+i x'_3}=-e^{2i\varphi}$$
$$x_2 x'_3-x_3 x'_2=0$$
so by Proposition \ref{prop:holonomy-fake-gauge}
$$\log\Hol(\gamma)=8 i\varphi \mod 2\pi i$$
which gives
$$\mathcal V_0=2\pi^2-4\pi\varphi\mod \pi^2.$$
\begin{Rem}
We know that for large genus, the surface looks like two great spheres intersecting with an angle $2\varphi$ and desingularized using a Scherk surface whose wings make an angle $2\varphi$. Moreover, one can check using the Weierstrass data computed in the proof of Point 5 of Proposition 19 of \cite{HHT2} that the normal of the limit Scherk surface points towards the wedge of angle $2\varphi$. Therefore, the limit enclosed volume is the complement of the volume of a wedge of angle $2\varphi$ in $\S^3$. 
\end{Rem}
\subsubsection{First order derivative}
We use $\dot{()}$ to denote derivatives with respect to $t$ at $t=0$.
By Proposition 31 in \cite{HHT2}, the first order derivatives of the parameters are 
\begin{equation}
\label{eq:centralder}
\begin{split}
\dot{x}_1&=i \left(\lambda ^2+1\right) \sin (2
   \varphi ) \log (\tan (\varphi ))
   \\
\dot{x}_2&=-2 \lambda ^2 \cos (\varphi ) \log
   (\cos (\varphi ))-2 \sin (\varphi )
   \sin (2 \varphi ) \log (\sin (\varphi
   ))-2 \cos (\varphi ) \cos (2 \varphi )
   \log (\cos (\varphi ))
 \\
 \dot{x}_3&=2 \lambda ^2 \sin (\varphi ) \log
   (\sin (\varphi ))-2 \sin (\varphi )
   \cos (2 \varphi ) \log (\sin (\varphi
   ))+2 \cos (\varphi )\sin (2 \varphi ) 
   \log (\cos (\varphi ))
\end{split}
\end{equation}
Let
\begin{equation*}
\begin{split}
f(t,\lambda)&=\log\left[ \frac{x_2'(t,\lambda)-i x_3'(t,\lambda)}{x_2'(t,\lambda)+i x_3'(t,\lambda)}\right]\\
g(t,\lambda)&=\frac{x_2 x'_3-x_3 x'_2}{1-x_1}.
\end{split}
\end{equation*}
Since $f(0,\lambda)$ does not depend on $\lambda$,
\begin{align*}
\frac{d}{dt}f(t,e^{i\theta(t)})\mid_{t=0}&=\pd{f}{t}(0,i)
=\left.\frac{\dot{x}'_2-i\dot{x}'_3}{x'_2-i x'_3}-\frac{\dot{x}'_2+i\dot{x}'_3}{x'_2+i x'_3}\right|_{\lambda=i}
\\
&=-2i\cos(\varphi)\dot{x}_2'+2i\sin(\varphi)\dot{x}_3'\mid_{\lambda=i}
=-8 \WWW.
\end{align*}
Regarding the second term in Proposition \ref{prop:holonomy-fake-gauge}, since $x_2 x'_3-x_3 x'_2=0$ at $t=0$, we have
$$\frac{d}{dt}\int_{\lambda_1(t)}^{\lambda_2(t)}
g(t,\lambda)\,d\lambda=\int_{-i}^i \pd{g}{t}(0,\lambda)\, d\lambda.$$
We find after simplification
\begin{equation}
\label{eq:x2dotx3p}
\dot{x}_2 x'_3+x_2\dot{x}'_3-\dot{x}_3 x'_2-x_3 \dot{x}'_2=
-\frac{(\lambda^2+1)^2}{\lambda^2}\WWW,
\end{equation}
thus
\begin{equation}
\label{eq:dgdt}
\pd{g}{t}(0,\lambda)=\frac{\dot{x}_2 x'_3-\dot{x}_3 x'_2+x_2\dot{x}'_3-x_3 \dot{x}'_2}{1-x_1}=2i\lambda^{-1}(\lambda+i)^2\WWW.
\end{equation}
Integrating both side then yields 
$$\frac{d}{dt}\int_{\lambda_1(t)}^{\lambda_2(t)}
g(t,\lambda)d\lambda
=(-8i + 2\pi)\WWW.$$
Altogether, we obtain by Proposition \ref{prop:holonomy-fake-gauge}
$$\frac{d}{dt}\log\Hol(\mathcal D,\gamma)\mid_{t=0}=
8\pi i\WWW
=\frac{i}{2}\dot{\mathcal W}.$$
\begin{Rem} Notice that the real terms cancel, as they should, since the holonomy is unitary.
\end{Rem}
By Equation \eqref{eq:volume},
$$\dot{\mathcal V}=\dot{\theta}\,\mathcal W_0+\frac{\pi}{4}\dot{\mathcal W}-\frac{\pi}{4}\dot{\mathcal W}=8\pi \dot{\theta}=16\pi \TTT.$$
\subsubsection{Second order derivative}
The second order derivatives of the parameters are:
\begin{align*}
\ddot{x}_1=&8 i \lambda  (\log (\tan(\varphi))
   \left[\cos ^2(\varphi ) \log (\cos (\varphi ))\left(\lambda ^2+4 \cos (2 \varphi
   )-1\right) +\sin ^2(\varphi ) \log (\sin
   (\varphi ))
   \left(-\lambda ^2+4 \cos (2 \varphi )+1\right)\right]\\
   &-\tfrac{i}{4} (\lambda^{-1}-\lambda) \ddot{\mathcal K}
\\
\ddot{x}_2=&4 \lambda  \sin (\varphi ) \log (\cos (\varphi ))
   \left[
   -2\cos ^2(\varphi ) \log (\cos (\varphi )) \left(\lambda^2-3\right)
   +\log (\sin (\varphi ))
   \left(\cos (2 \varphi )\left(\lambda ^2-3\right)+3\lambda ^2-1\right)
   \right]\\
   &+\tfrac{1}{4} \sin(\varphi)(\lambda^{-1}+\lambda) \ddot{\mathcal K}
   \\
\ddot{x}_3=&4 \lambda  \cos (\varphi ) \log (\sin (\varphi )) \left[-2
   \sin ^2(\varphi ) \log (\sin
   (\varphi ))\left(\lambda ^2-3\right)+\log (\cos (\varphi ))\left(-\cos (2 \varphi
   )\left(\lambda ^2-3\right)+3 \lambda ^2-1\right)\right]\\
   &+\tfrac{1}{4} \cos(\varphi)(\lambda^{-1}+\lambda) \ddot{\mathcal K}.
\end{align*}
They are computed in Appendix D of \cite{HHT2}.
\begin{Rem}We took into account that the parameters have been scaled by $\frac{1}{\sqrt{\mathcal K}}=\frac{t}{s}$ which adds the $\ddot{\mathcal K}$ terms. In fact, the terms $\ddot{\mathcal K}$ cancel in our computation so the actual value of $\ddot{\mathcal K}$ does not matter here.
We have $\dot{\mathcal K}=0$ so the scaling does not affect the first order derivatives.
\end{Rem}

Since $f(0,\lambda)$ does not depend on $\lambda$,
\begin{align*}
\frac{d^2}{dt^2}f\left(t,e^{i\theta(t)}\right)\mid_{t=0}
&=\frac{\partial^2 f}{\partial t^2}(0,i)-2\frac{\partial^2 f}{\partial \lambda\partial t}(0,i)\,\theta'\\
\frac{\partial^2 f}{\partial t^2}(0,i)
&=\frac{\ddot{x}'_2-i\ddot{x}'_3}{x'_2-i x'_3}
-\frac{(\dot{x}'_2-i \dot{x}'_3)^2}{(x'_2-i x'_3)^2}
-\frac{\ddot{x}'_2+i\ddot{x}'_3}{x'_2+i x'_3}
+\frac{(\dot{x}'_2+i \dot{x}'_3)^2}{(x'_2+i x'_3)^2}
\\
&=-2i\cos(\varphi)\ddot{x}'_2+2i\sin(\varphi)\ddot{x}'_3-2i\sin(2\varphi)(\dot{x}'_2)^2
+2i\sin(2\varphi)(\dot{x}'_3)^2-4i\cos(2\varphi)\dot{x}'_2\dot{x}'_3
\end{align*}
\begin{align*}
\frac{\partial^2 f}{\partial\lambda\partial t}(0,i)
&=\frac{\dot{x}''_2-i\dot{x}''_3}{x'_2-i x'_3}
-\frac{(\dot{x}'_2-i \dot{x}'_3)(x''_2-i x''_3)}{(x'_2-i x'_3)^2}
-\frac{\dot{x}''_2+i\dot{x}''_3}{x'_2+i x'_3}
+\frac{(\dot{x}'_2+i \dot{x}'_3)(x''_2+i x''_3)}{(x'_2+i x'_3)^2}
\\
&=-2i\cos(\varphi)\dot{x}''_2+2 i \sin(\varphi) \dot{x}''_3-2\cos(\varphi) \dot{x}'_2+2\sin(\varphi)\dot{x}'_3,
\end{align*}
where everything is evaluated at $\lambda=i$.
We find after simplification
\begin{align*}&-\cos(\varphi)\ddot{x}'_2+\sin(\varphi)\ddot{x}'_3=24\,\WWW\TTT\\
&-\cos(\varphi)\dot{x}''_2+\sin(\varphi)\dot{x}''_3=4\,\WWW\\
&-\sin(2\varphi)(\dot{x}'_2)^2+\sin(2\varphi)(\dot{x}'_3)^2-2\cos(2\varphi)\dot{x}'_2\dot{x}'_3=-16\,\WWW\TTT\\
&-\cos(\varphi)\dot{x}_2+\sin(\varphi)\dot{x}_3=4i\,\WWW.\end{align*}
Putting everything together we have
$$\frac{d^2}{dt^2}f\left(t,e^{i\theta(t)}\right)\mid_{t=0}=-48 i \, \WWW\TTT.$$
Since $g(0,\lambda)$ does not depend on $\lambda$,
$$\frac{\partial^2}{\partial t^2}\int_{\lambda_1(t)}^{\lambda_2(t)}g(t,\lambda)d\lambda\mid_{t=0}
=\int_{-i}^i \frac{\partial^2 g}{\partial t^2}(0,\lambda) d\lambda
+2\pd{g}{t}(0,i)\dot{\lambda}_2-2\pd{g}{t}(0,-i)\dot{\lambda}_1.$$
By Equation \eqref{eq:dgdt}, we have
$$2\pd{g}{t}(0,i)\dot{\lambda}_2-2\pd{g}{t}(0,-i)\dot{\lambda}_1
=16\WWW \dot{\theta}=32\WWW\TTT.$$
We compute the other term:
$$\frac{\partial^2 g}{\partial t^2}=\frac{1}{1-x_1}\left(
\ddot{x}_2 x'_3+2 \dot{x}_2\dot{x}_3'+x_2\ddot{x}'_3-\ddot{x}_3 x'_2-2\dot{x}_3\dot{x}'_2-x_3\ddot{x}'_3\right)
+\frac{2\,\dot{x}_1}{(1-x_1)^2}\left(
\dot{x}_2 x'_3+x_2 \dot{x}'_3-\dot{x}_3 x'_2 -x_3\dot{x}'_2\right).
$$
We find after simplification:
\begin{align*}&\ddot{x}_2 x'_3-\ddot{x_3} x'_2=-2\frac{(\lambda^2-1)(\lambda^2-3)}{\lambda}\WWW\TTT\\
&x_2\ddot{x}'_3-x_3\ddot{x}'_2=6\frac{(\lambda^4-1)}{\lambda}\WWW\TTT\\
&\dot{x}_2\dot{x}'_3-\dot{x}_3\dot{x}'_2=-8\lambda \WWW\TTT
\end{align*}
which gives
$$\ddot{x}_2 x'_3+2 \dot{x}_2\dot{x}_3'+x_2\ddot{x}'_3-\ddot{x}_3 x'_2-2\dot{x}_3\dot{x}'_2-x_3\ddot{x}'_3=4\frac{(\lambda^2+1)(\lambda^2-3)}{\lambda}\WWW\TTT.$$
Using Equation \eqref{eq:x2dotx3p}
$$\left(
\dot{x}_2 x'_3+x_2 \dot{x}'_3-\dot{x}_3 x'_2 -x_3\dot{x}'_2\right)\dot{x}_1
-i\frac{(\lambda^2+1)^3}{\lambda^2}\WWW\TTT.$$
Everything together gives
\begin{align*}\frac{\partial^2 g}{\partial t^2}&=-16(\lambda+i)\WWW\TTT\\
\int_{-i}^i \frac{\partial^2 g}{\partial t^2}d\lambda&=32\WWW\TTT.
\end{align*}
Finally, we obtain
$$\frac{d^2}{dt^2}\log\Hol(\mathcal D,\gamma)\mid_{t=0}
=4\, \WWW\TTT\left(-48\, i + 32\, i +32\, i\right)=64\, i\,\WWW\TTT.$$
Using that $\ddot{\theta}=0$, we obtain from \eqref{eq:volume}
$$\ddot{\mathcal V}=2\dot{\theta}\dot{\mathcal W}-\frac{\pi}{2i}\frac{d^2}{dt^2}\log\Hol(\gamma)=64\pi\WWW\TTT-32\pi\WWW\TTT=32\pi\WWW\TTT.$$
This concludes the proof of Theorem \ref{theorem:volume-order2}.

\subsection{A gauge theoretic computation of the holonomy $\Hol(\mathcal D,\gamma)$ in terms of ${\bf x}=(x_1,x_2,x_3)$}
In the last section, we present an alternate method for calculating the holonomy of $\mathcal D$ along the curve $\gamma$ in the case of Lawson-type CMC surfaces as in Section \ref{sec:LCMC} 
without using regularizations. Instead, we will use explicit Darboux coordinates of moduli spaces of flat connections on the 4-punctured sphere.

\subsubsection{The holomorphic symplectic form and the Atiyah-Bott form}
We use the notations from Section \ref{sec:LCMC} and fix $s=\tfrac{1}{2g+2}$, for $g\in\N_{>1}$. Consider the $\Z_{g+1}$-action on the moduli space $\mathcal M$ of flat totally reducible $\mathrm{SL}(2,\C)$-connections on $\Sigma_{g}$
\[\sigma\colon\mathcal M\to\mathcal M;\quad\sigma([\nabla])=[\sigma^*\nabla],\]
where
$\sigma\colon \Sigma_g\to\Sigma_g$ is the generator 
of the natural $\Z_{g+1}$-action on $\Sigma_{g}$, i.e.,
\[(y,z)\mapsto (e^{4 i \pi s}y,z).\] After removing reducible gauge classes, the fix point set of $\sigma$ consists of
different connected components, labeled by integers $l_j=0,1,\dots,g$ for $j=1,\dots,4.$ Each component $\mathcal C_{l_1,\dots,l_4}$ is isomorphic (after 
appending the reducible gauge classes) to the moduli space of totally reducible  logarithmic connections on the 4-punctured sphere with eigenvalues
$\pm\tfrac{l_j}{2g+2}$ of the residues at  the singular points $p_j.$ By definition, a totally reducible connection is either irreducible or the direct sum of irreducible connections.
As we are only interested in immersions $f$, we mainly restrict to the component $\mathcal C_1=\mathcal C_{1,1,1,1}$
(i.e., $l_1=\dots=l_4=1$), see \cite[Theorem 3.3]{HHSch},
and identify $\mathcal C_1$ (without further indication) with the corresponding moduli space of totally reducible logarithmic connections on $\CP^1$
with singular points $p_1,\dots,p_4$ such that the eigenvalues of the residues at each $p_j$ are $\pm s$.
 
 Recall that we only consider gauge classes of unitary connections, since $\nabla^\lambda$ is unitary along $\lambda\in\mathbb S^1.$ They are in one-to-one correspondence with stable parabolic structures. The case we will mostly interested in is $l=1$ and $g>1$, where the parabolic weight is $\tfrac{l}{2g+2}<\frac{1}{4}.$ Then by \cite[Lemma 10]{Fuchsian} unstable parabolic structure appears if and only if the logarithmic connection is non-Fuchsian, i.e., when the underlying holomorphic structure of the rank two bundle is non-trivial. (For details about parabolic structures, stability and the relationship with flat
unitary connections see \cite{Biswas} or \cite{HeHe,Fuchsian} and the references therein). This implies that the unitary
connections on the 4-punctured sphere we are interested in lie in the gauge class of Fuchsian systems.
From $\mathcal C_1$ we therefore remove the non-Fuchsian locus to obtain
$\mathcal C_1^s.$
Consider the Atiyah-Bott  symplectic form $\Omega$  (see \cite{AB} or Goldman \cite{Gold}) restricted to the equivariant component
$\mathcal C_1^s$. It is well-known (see e.g. \cite{RSW} or \eqref{cur-CS} above) that  $\Omega$ is the curvature
of $\mathcal D$ on $\mathcal L\to\mathcal M.$

\begin{Lem}\label{lem:4fQC1}
There is a 4-fold covering 
\[\mathcal Q:=\{(x_1,x_2,x_3)\in\C^3\mid x_1^2+x_2^2+x_3^2=1\}\to \mathcal C_1^s\]
branched over the 3 reducible gauge classes. 
\end{Lem}
\begin{proof}
The map is  given by
\[(x_1,x_2,x_3)\mapsto [d+\eta_{\bf x}]\]
where $\eta_{\bf x}=\eta$ is given as in \eqref{eq:eta} with $A_k$ as in \eqref{eq:etaAk}  rescaled according to Remark \ref{etarescaled}  with $s=\tfrac{1}{2g+2}$. That this is a 4-fold covering branched only over the 3 reducible connections 
is shown in \cite[Lemma 16]{Fuchsian}. The (generically)
4 preimages are related to each other through
 conjugation with $C_2,C_3,C_4$ in \eqref{eq:Cs}.
\end{proof}

Let the (immersed) CMC surface $f\colon\Sigma=\Sigma_g\to\mathbb S^3$ be given in terms of the Fuchsian potential
$d+\eta$ as in Section \ref{sec:dpwbasic}.
This means that  the pull-back of $d+\eta$ via the $(g+1)$-fold totally branched covering $\Sigma_g\to\CP^1$
is gauge equivalent (by a holomorphic family of gauge transformations $h=h(\lambda)$ which extends holomorphically to
$\lambda=0$) to the associated family $\nabla^\lambda$ of flat connections of $f$.
(Note that the gauge transformations $h$ are only well-defined up to $\pm1$, which can be resolved easily by either
tensoring with a flat $\Z_2$-bundle or by working on a 2-fold branched covering and can therefore be ignored in the following.)
In particular, this implies that the connections $\nabla^\lambda$ are contained in the equivariant component $\mathcal C_1$. The corresponding curve
\[\gamma\colon [0,c]\to \mathcal C^s_1\subset \mathcal M\] in \eqref{def:gamma} is by construction closed, but 
 its lift to
the quadric
$\mathcal Q$ via Lemma \ref{lem:4fQC1} is not closed
by Proposition \ref{prop:diagonal}.

It is therefore more appropriate to use other coordinates $(u,r)$ on some open and dense subset $\mathcal U$ of $\mathcal C_1^s$  which can then be expressed in terms of the parameter vector $\bf x$.  Such coordinates exists by \cite{LorS}. We use the same normalization as
in \cite{Fuchsian}, see also
 \cite{HeHe}. Consider
 
 \begin{equation}\label{eq:tildeAs}
 \begin{split}
 \wt A_1&:=\begin{pmatrix}-1&2u\\0&1 \end{pmatrix}+r\begin{pmatrix}-u&u^2\\-1&u \end{pmatrix}\;,\quad
 \wt A_2:=\begin{pmatrix}-1&0\\-2&1 \end{pmatrix}+r\begin{pmatrix}0&0\\1-u&0 \end{pmatrix}\;,\\
 \wt A_3&:=\begin{pmatrix}1&0\\2&-1 \end{pmatrix}+r\begin{pmatrix}u&-u\\u&-u \end{pmatrix}\;,\quad
 \wt A_4:=\begin{pmatrix}1&-2u\\0&-1 \end{pmatrix}+r\begin{pmatrix}0&u-u^2\\0&0 \end{pmatrix}\\
 \end{split}
\end{equation}
and
\begin{equation}\label{eq:tildeeta}\wt \eta(u,r)=s\sum_{k=1}^4\wt A_k\frac{dz}{z-p_k}.\end{equation}

If
the constraint
\[x_1^2+x_2^2+x_3^2=1\] is satisfied,
then there exist a  matrix \begin{equation*}
\begin{split}
&K=K(x_1,x_2,x_3)=\\
&\sqrt{\frac{\left(x_1^2+x_3^2\right) (x_2+i x_3)}{2 (x_1-1) (x_1-ix_2 x_3)}}
\left(
\begin{array}{cc}
 \tfrac{1- x_1}{x_2+i x_3} & \frac{(1-x_1) (x_1-i x_2 x_3)}{\left(x_1^2+x_3^2\right) (x_2+i x_3)} \\
 1 & \frac{i x_2 x_3-x_1}{x_1^2+x_3^2} \\
\end{array}
\right)\in\mathrm{SL}(2,\C)\end{split}\end{equation*}

 with
\[(d+\eta_{(x_1,x_2,x_3)}).K=d+\wt\eta(u,r)\]
where
\begin{equation}\label{eq:usx123}
\begin{split}
u&=-\frac{(x_2+ i x_1 x_3)^2}{(i x_1 x_2-x_3)^2}\\
r&=\frac{2 x_2 (-i x_1 x_2+x_3)^2}{\left(x_2^2+x_3^2\right) (x_2+i x_1 x_3)}.\\
\end{split}
\end{equation}

It is well-known that the symplectic form $\Omega$ can be computed in terms of the residues, if the gauge classes
are represented by logarithmic connections, see for example \cite{AlMa} or \cite{Audin}. We
then have the following lemma.
\begin{Lem}\label{Lem:omus}
There is an open dense subset $\mathcal U\subset \mathcal C_1^s$
such that $(u,r)$ are coordinates on $\mathcal U.$ Moreover,  
\begin{equation}\label{eq:Omegadsdu}\Omega=\tfrac{1}{2}dr\wedge du=\tfrac{1}{2} d(rdu).\end{equation}
\end{Lem}
\begin{proof}
We give a direct proof adapted to the situation at hand. A more systematic proof can be given following \cite{Audin},
see for example \cite[Section 3.4]{HHTSD}.
By construction, $(u,r)$ are well-defined and injective on some open and dense subset of the Fuchsian locus $\mathcal C_1^s$, which itself is an open dense subset of $\mathcal C_1$. That $(u,r)$ are actually coordinates follows from our derivation of \eqref{eq:Omegadsdu} below.

Let $d+\wt\eta(u_0,r_0)$ be given.
Consider the branched  $(g+1)$-fold covering $\pi\colon\Sigma\to\CP^1$, and the pull-back $\wt\nabla=\pi^*(d+\wt\eta(u_0,r_0))$.
Note that the eigenvalues of the residues of $\wt\nabla$ are $\pm\tfrac{1}{2}.$ In order to gauge the logarithmic singularities away, we therefore need the square root of a holomorphic coordinate centered at the singular points.
Globally, we
have a `double-valued' gauge transformation 
\[h=h_0\in\Gamma(\Sigma\setminus\pi^{-1}(p_1,\dots,p_4), \mathrm{SL}(2,\C)\otimes L)\]
where $L\to \Sigma\setminus\pi^{-1}(p_1,\dots,p_4)$ is a  line bundle  with fixed trivialization
$L^2\equiv\underline\C$ such that 
\[\nabla:=\wt\nabla.h\] is smooth globally on $\Sigma.$ (Likewise, one could perform all computations on a 
2-fold covering $\Sigma\to\Sigma_g$ on which $h$ is simple-valued. In fact, the $\Z_2$-valued monodromy of $L$
is then just the covering monodromy of $\Sigma\to\Sigma_g$.)

Let 
\[U:=\frac{\partial\wt \eta(u,r)}{\partial u}_{\mid (u_0,r_0)},\,S:=\frac{\partial\wt \eta(u,r)}{\partial r}_{\mid (u_0,r_0)}\in H^0(\CP^1\setminus\{p_1,\dots,p_4\},K\otimes\mathfrak{sl}(2,\C))\] be two meromorphic 1-forms with first order 
poles which are tangent to $\mathcal M$ at $d+\wt\eta(u_0,r_0).$
By assumption, the conjugacy classes of the residues of $\wt \eta(u,r)$ are independent of $u$ and $r$. Therefore  we can express
\[U=\sum_{l=1}^4 [U_l,\Res_{p_l}\wt\eta(u_0,r_0)]\frac{dz}{z-p_l}\quad\text{and}\quad
S=\sum_{l=1}^4 [S_l,\Res_{p_l}\wt \eta(u_0,r_0)]\frac{dz}{z-p_l}\]
for $U_l,S_l\in\mathfrak{sl}(2,\C)$, $l=1,\dots,4.$
Of course, $U_l$ and $S_l$ are only determined up to addition 
with multiples of Res$_{p_l}\eta(u_0,r_0).$

Let $\rho_l\colon\Sigma_g\to\R$, $l=1,\dots,4$, be smooth functions with small support centered at $\pi^{-1}(p_l)$ and
constant to $1$ in a small neighborhood of $\pi^{-1}(p_l)$.
Varying $u$ and $r$ respectively, we then find two 1-parameter families of (double-valued) gauge transformations
of the form
\[h_u=(\Id+\sum_l\rho_lU_l (u-u_0)+o(u-u_o)\,)g\]
and
\[\wt h_r=(\Id+\sum_l\rho_lS_l (r-r_0)+o(r-r_o)\,)g\]
such that
\[\pi^*(d+\wt\eta(u,r_0)).h_u\quad\text{and}\quad \pi^*(d+\wt\eta(u_0,r)).\wt h_r\]

are globally smooth on $\Sigma$ for $u\sim u_0$ respectively $r\sim r_0.$ 
Define
\[A=h^{-1}(\pi^*\sum_l\rho_lU_l )\, h\quad\text{and}\quad B=h^{-1}(\pi^*\sum_l\rho_lS_l)h.\]
Note that  $\Res_{p_l}\wt\eta(u_0,r_0)=\tfrac{1}{2}\wt A_l.$
Using $s=\tfrac{1}{2g+2}$ and expanding in $u$ or $r$, respectively,
gives the smooth  tangent vectors $X,Y\in T_\nabla\mathcal C_1\subset T_\nabla\mathcal M$
\begin{equation*}
\begin{split}X:=&\tfrac{\partial}{\partial u}=d^\nabla A +h^{-1}\pi^*Uh=h^{-1}(\sum_{l}(d+\sum_n\tfrac{1}{2k}\wt A_n\pi^*\tfrac{dz}{z-p_n})(\rho_lU_l)+[U_l,\tfrac{1}{2k}\wt A_l]\pi^*\tfrac{dz}{z-p_l})h\\
&=h^{-1}(\sum_l d\rho_l U_l+[U_l,\tfrac{1}{2k}\wt A_l]\pi^*\tfrac{dz}{z-p_l}+\sum_n [\tfrac{1}{2k}\wt A_n,U_l]\rho_l \pi^*\tfrac{dz}{z-p_n})h
\end{split}
\end{equation*}
and 
\begin{equation*}
\begin{split}Y:=\tfrac{\partial}{\partial r}=d^\nabla B+h^{-1}\pi^*Sh
&=h^{-1}(\sum_j d\rho_j S_j+[S_j,\tfrac{1}{2k}\wt A_j]\pi^*\tfrac{dz}{z-p_j}+\sum_m \rho_j[\tfrac{1}{2k}\wt A_m,S_j] \pi^*\tfrac{dz}{z-p_m})h.
\end{split}
\end{equation*}
Recall that on the closed Riemann surface $\Sigma_g$, for a flat connection $\nabla$ and for tangent vectors represented by $X,Y\in\ker(d^\nabla)$ 
we have 
\[\Omega(X,Y)=\frac{i}{2\pi}\int_\Sigma\tr(X\wedge Y)\]
independently of the (smooth) representation.
Furthermore, we can assume without loss of generality  that  
$\pi^{-1}(\infty)\cap\mathrm{supp}(\rho_l)=\emptyset=\mathrm{supp}(\rho_l)\cap\mathrm{supp}(\rho_j)$ for all
$j\neq l=1,\dots,4.$ In particular, this gives 
\[\int_\Sigma d\rho_l\wedge \rho_l\pi^*\tfrac{dz}{z-p_j}=\int_\Sigma d\rho_l\wedge\pi^*\tfrac{dz}{z-p_j}=0=\int_\Sigma d\rho_l\wedge \rho_j\pi^*\tfrac{dz}{z-p_m}\]
for all $l\neq j$. Moreover, we have for all $l=1,\dots,4$
\[\int_\Sigma d\rho_l\wedge (1-\rho_l)\pi^*\tfrac{dz}{z-p_l}=\int_\Sigma (d\rho_l- \tfrac{1}{2}d\rho_l^2)\wedge \pi^*\tfrac{dz}{z-p_l}=-\pi i (g+1)\]
(where the factor $g+1$ comes from the covering).
Overall integration then yields (using the invariance of the trace under conjugation)
\begin{equation}
\begin{split}
\Omega(X,Y)&=\frac{i}{2\pi}\int_\Sigma\tr(X\wedge Y)
=\frac{i}{2\pi}\int_\Sigma\sum_{l}\tr(U_l[S_l,\tfrac{1}{2g+2}\wt A_l]-S_l[U_l,\tfrac{1}{2g+2}\wt A_l])d\rho_l\wedge (1-\rho_l)\pi^*\tfrac{dz}{z-p_l}\\
&=-(g+1)\sum_l\tr(U_l[\tfrac{1}{2g+2}\wt A_l,S_l])=-\tfrac{1}{2}\sum_l\tr(U_l[\wt A_l,S_l]).
\end{split}
\end{equation}
From \eqref{eq:tildeAs} we therefore compute
\[\Omega(X,Y)=-\tfrac{1}{2}(1+0+0+0)=-\tfrac{1}{2}.\]
In particular, $X,Y$ are pointwise linearly independent as tangent vectors and give (local) coordinates $(u,r)$ on $\mathcal U\subset\mathcal C_1^s$ via \eqref{eq:tildeeta}.
Finally, since $X=\tfrac{\partial}{\partial u}$ and $Y=\tfrac{\partial}{\partial r}$, we have $\Omega=\tfrac{1}{2}dr\wedge du.$
\end{proof}
Note that the $(u,r)$-coordinates cover an open dense subset \[\mathcal U:=\{(u,r)\mid u\in\C\setminus\{0,1\},r\in \C\}\subset \mathcal C_1^s,\] but become singular exactly over 
the totally reducible (and unitary) connections.
This clearly happens at the Sym points $\lambda_1$ and $\lambda_2$, so that we have to study the behavior at those points.

\begin{Pro}\label{Pro7}
If $\gamma\colon [0,c]\to\mathcal C_1^s$ is given by
\[(x_1,x_2,x_3)\colon[0,c]\to \mathcal Q,\]
then $-\tfrac{1}{2} rdu$ pulls back to a smooth 1-form on $[0,c]$ 
and we have
\begin{equation}\label{eq:hol-intsdu}
\Hol(\mathcal D,\gamma)=\exp{(-\tfrac{1}{2} \int_\gamma rdu)}.
\end{equation}
\end{Pro}
\begin{proof}
The parameter $u$ encodes the parabolic structure, which is closely related to
unitary connections by the Mehta-Seshadri theorem. 
For details about parabolic structures, see for example \cite{Biswas} or \cite{Fuchsian} and the references therein. In particular, 
any Fuchsian system with unitary monodromy has a strictly semi-stable parabolic structure if and only if the monodromy is reducible. Strictly semi stable parabolic structure happens exactly when $u$ becomes $0,1$ or $\infty$, see for example \cite[Section 2.4]{Fuchsian} or \eqref{eq:semistable} below. Recall that we only work on the unit circle, and all connections are unitary (up to conjugation).
Hence, 
 we obtain (keeping in mind the unitarity) 
\begin{equation}\label{eq:semistable}
\begin{split}
u=0 \quad &\Longleftrightarrow \quad x_1=x_2=0\quad\text{and}\quad x_3=\pm1\\
u=1 \quad &\Longleftrightarrow \quad x_2=x_3=0\quad\text{and}\quad x_1=\pm 1\\
u=\infty \quad &\Longleftrightarrow \quad x_1=x_3=0\quad\text{and}\quad x_2=\pm1\;.\\
\end{split}
\end{equation}

Note that for  the Lawson surfaces and their CMC deformations, we have for large genus $g$ corresponding to $s=\tfrac{1}{2g+2}\sim0$
that $u$ does not become $0$ or $\infty$ along the unit circle by Proposition \ref{prop:diagonal} and the proof of Proposition \ref{prop10}. 
\begin{Rem}
It would be interesting to have an example where $u$ becomes $0$ or $\infty$ along the unit circle, since that would give a new additional closed CMC surface on a slightly different covering of $\CP^1$. 
\end{Rem}

Thus we only consider the case $u=1$, the other cases work similarly. Locally, near $x_2=x_3=0$, $x_1=\pm\sqrt{1-x_2^2-x_3^2},$ and we compute
\begin{equation}\label{consdu}
\tfrac{1}{2}rdu=\frac{ 2x_2(x_1-i x_2 x_3)}{x_1 \left(x_2^2-1\right)} dx_2+\frac{2 ix_2}{x_1}dx_3.\end{equation}

We normalize $d+\eta_{(x_1,x_2,x_3)}$ by conjugating with 
\[h=\left(
\begin{array}{cc}
 1 & i (x_1+x_3) \\
 1 & -i (x_1+x_3) \\
\end{array}
\right)\]
to obtain
\begin{equation}
\begin{split}
(d+\eta_{(x_1,x_2,x_3)}).h=d+&s \left(
\begin{array}{cc}
 x & 1-x^2 \\
 1 & -x \\
\end{array}
\right)\frac{dz}{z-p_1}
+s\left(
\begin{array}{cc}
 -x & \frac{x^2-1}{y} \\
 -y & x \\
\end{array}
\right) \frac{dz}{z-p_2} \\
+&s \left(
\begin{array}{cc}
 -x & \frac{1-x^2}{y} \\
 y & x \\
\end{array}
\right)\frac{dz}{z-p_3} 
+s  \left(
\begin{array}{cc}
 x & x^2-1 \\
 -1 & -x \\
\end{array}
\right)\frac{dz}{z-p_4} \end{split}\end{equation} with
\begin{equation}\label{eq:xyx1x2x3}x=x_2\quad \text{and}\quad y=\frac{x_1-i x_3}{x_1+i x_3}.\end{equation}
As in the proof of Lemma \ref{Lem:omus} one can show that in terms of $(x,y)$
\[\Omega=-\frac{1}{y}dx\wedge dy=-d(\frac{x}{y}dy).\]
Using \eqref{eq:xyx1x2x3} and $1=x_1^2+x_2^2+x_3^2$ we obtain
\[-\frac{x}{y}dy=-\frac{2 ix_2^2 x_3}{x_1 \left(x_2^2-1\right)} dx_2+\frac{2 ix_2}{x_1}dx_3.\]
Therefore, the two connection 1-forms $\tfrac{1}{2}sdu$ and $-\frac{x}{y}dy$ differ by
\[d\log(x_2^2-1),\]
which is smooth and well-defined near
\[(x_1,x_2,x_3)=\pm(1,0,0)\]
and has periods in $2\pi i\Z$ only.
 Hence, the two connection 1-forms $\tfrac{1}{2}rdu$ and $-\frac{x}{y}dy$ give rise to the same
 parallel transport on a neighborhood of $(x_1,x_2,x_3)=\pm(1,0,0).$ This computation shows that  $\mathcal L$ trivializes over $\mathcal U$ such that the connection 1-form
of $\mathcal D$ is $\tfrac{1}{2}rdu,$ which proves  the proposition.
\end{proof}
\begin{Rem}
It should be  noted that the line bundle $\mathcal L\to\mathcal C_1^s$ is non-trivial. In fact, investigating the behavior of $\Omega$ at $u=0,1,\infty$
in detail gives a way to compute the symplectic volume of the unitary locus, resulting in the very particular instance of Witten's formula \cite{Witten} for the 4-punctured sphere.
\end{Rem}

From Proposition \ref{Pro7} and \eqref{consdu} we then obtain
\begin{The}\label{thm6}
Let $f\colon \Sigma_{g+1}\to\mathbb S^3$ be a closed CMC surface determined by the  Fuchsian potential $d+\eta$,
where ${\bf x}=(x_1,x_2,x_3)$ with $x_j\colon \{\lambda\in\C\,\mid \,0<\lambda\bar\lambda<1+\epsilon\}\to\C$.
Then
\begin{equation*}
\Hol(\mathcal D,\gamma)=\exp \left[ \int_\gamma \left(-\frac{2x_2(x_1-i x_2x_3)}{x_1 (x_2^2-1)} dx_2-\frac{2 ix_2}{x_1}dx_3\right)\right]
\end{equation*}
where $\gamma\colon[\tau_1,\tau_2]\to\C^*,\tau\mapsto e^{i\tau}$.\end{The}
\begin{Rem}
The formulas in Theorem \ref{The:LeV} and Theorem \ref{thm6} differ because they are derived with respect to different
trivializations of $\mathcal L\to C^s_1.$ In particular, in the trivialization used for Theorem  \ref{The:LeV} the two end points of $\gamma= [\nabla^\lambda]$ at $\lambda_1$ and $\lambda_2$ are lifted to different connections, causing the boundary term $4\log[\tfrac{x_2'(\lambda_2)-i x_3'(\lambda_2)}{x_2'(\lambda_2)+i x_3'(\lambda_2)}]$ in \eqref{eq:LEV}. 
\end{Rem}



\begin{thebibliography}{10}
\bibitem{AB}{ M. F. Atiyah, R. Bott}, {\em  The Yang-Mills Equations over Riemann Surfaces},
Phil. Trans. R. Soc. Lond. A (1983) 308, 523--615

\bibitem{AlMa}{A. Alekseev, A. Malkin}, {\em Symplectic Structure of the Moduli Space of Flat Connections on a Riemann surface}, { Comm. Math. Phys.},
Volume 169, 99--119 (1995).

\bibitem{Audin}{M.  Audin}, {\em  Lectures on gauge theory and integrable systems}, In: Hurtubise, J., Lalonde, F., Sabidussi, G. (eds) Gauge Theory and Symplectic Geometry. NATO ASI Series, vol 488. Springer, Dordrecht. 

\bibitem{Biswas} I. Biswas, 
{\em Parabolic bundles as orbifold bundles,}
 Duke Math. J. 88, no. 2, 305--325 (1997). 
\bibitem{B} A. I. Bobenko, {\em  All constant mean curvature tori in $\R^3$, $\mathbb S^3$ ,$\mathbb H^3$ in terms of theta-functions}. Math. Ann. 290 (1991), no. 2, 209-245.

\bibitem{Bo} A. I. Bobenko, {\em Constant mean curvature surfaces and integrable equations}, Uspekhi Mat Nauk 46:4 (1991), 3--42.

\bibitem{BHS1} A. I. Bobenko, S. Heller, N. Schmitt, {\em Constant mean curvature surfaces based on fundamental quadrilaterals}, Math. Phys. Anal. Geom. 24 (2021), no. 4, Paper No. 37.

\bibitem{BHS2} A. I. Bobenko, S. Heller, N. Schmitt, {\em Minimal reflection surfaces in $\mathbb S^3$. Combinatorics of curvature lines and minimal surfaces based on fundamental pentagons},  Journal of Geometry and Physics
\href{https://doi.org/10.1016/j.geomphys.2024.105407}{https://doi.org/10.1016/j.geomphys.2024.105407} (2025).

\bibitem{HHT2} S. Charlton, L. Heller, S. Heller, M. Traizet, {\em Minimal surfaces and alternating multiple zetas},  \href{https://arxiv.org/abs/2407.07130}{preprint:  arXiv:2407.07130}
 
 

\bibitem{Chern-Simons}
S.S. Chern, J. Simons, {\em  Some cohomology classes in principal fiber bundles and their application to riemannian geometry},
Proc. Nat. Acad. Sci. U.S.A. 68 (1971), 791--794.



\bibitem{DPW} J. Dorfmeister, F. Pedit, H. Wu, {\em   Weierstrass type representation of harmonic maps into symmetric spaces}. Comm. Anal. Geom.  6 (1998),  no. 4, 633-668.

 
\bibitem{Freed} D. Freed, {\em Classical Chern-Simons theory I}, Adv. Math. 113 (1995), 237--303.

\bibitem{Gold} W. Goldman,  {\em The symplectic nature of fundamental groups of surfaces.} Advances in Mathematics, 52(2):200--225, 1984
  
  \bibitem{HPRR}
L.  Hauswirth, J. Perez, P. Romon, A. Ros, {\em 
The periodic isoperimetric problem}, 
Trans. Am. Math. Soc. 356, No. 5, 2025--2047 (2004).

 \bibitem{HeHe} L. Heller, S. Heller, {\em Abelianization of Fuchsian systems and Applications}. Jour. Symp. Geo. 14(4),
1059--1088 (2016).

  \bibitem{HHSch} L. Heller, S. Heller, N. Schmitt {\em Navigating the Space of Symmetric CMC Surfaces}. J. Diff. Geom.,
110(3), 413--455 (2018).

 \bibitem{HHT1} L. Heller, S. Heller, M. Traizet, {\em Area estimates for high genus Lawson surfaces via DPW}, J. Diff. Geom., 124 (1), 1--35, (2023).
 
 \bibitem{Fuchsian} L. Heller, S. Heller, {\em Fuchsian DPW potentials for Lawson surfaces}, Geom. Dedicata, 217, (2023).

  \bibitem{HHTSD} L. Heller, S. Heller, M. Traizet, {\em Loop group methods for the non-abelian Hodge correspondence on a 4-punctured sphere}, Math. Annalen (2025), \href{https://doi.org/10.1007/s00208-025-03165-y}{https://doi.org/10.1007/s00208-025-03165-y}. 
  \bibitem{He1}  S. Heller, {\em  Minimal surfaces and stable bundles},
J. Reine
Angew. Math., Volume 685 (2013), 105--122.
 
 \bibitem{He3}  S. Heller, {\em  A spectral curve approach to Lawson symmetric CMC surfaces of
  genus $2$}, Math. Annalen, Volume 360, Issue 3 (2014), Page 607--652.
 
 \bibitem{Hi} N. J. Hitchin, {\em  Harmonic maps from a $2$-torus to the $3$-sphere}. J. Diff. Geom.  31 (1990), no. 3, 627-710. 
 
  \bibitem{Hi2} N. J. Hitchin, {\em  The Wess-Zumino term for a harmonic map}, J. reine angew. Math., Volume 543 (2002), 83--101.
   
  \bibitem{SKKR} N. Schmitt, M. Kilian, S.-P. Kobayashi, W. Rossman, {\em  Unitarization of monodromy representations and constant mean curvature trinoids in 3-dimensional space forms}, J. Lond. Math. Soc. (2) 75 (2007).
\bibitem{L} H. B. Lawson, {\em  Complete minimal surfaces in $S\sp{3}$}, Ann. of Math. (2) 92 (1970), 335-374.

\bibitem{LorS}  F. Loray and M. H. Saito, {\em Lagrangian fibration in duality on moduli space of rank two logarithmic connections over 
the projective line}, IMRN Volume 2015, Issue 4, 2015, Pages 995--1043.

\bibitem{RSW} T. R. Ramadas, M. Singer,  J. Weitsman, {\em Some Comments on Chern-Simons Gauge Theory}, Commun. Math. Phys. 126, 409--420 (1989).

\bibitem{Ros} A. Ros, {\em 
Isoperimetric inequalities in crystallography}, J. Am. Math. Soc. 17, No. 2, 373--388 (2004).
\bibitem{Witten89}  E. Witten, {\em 
Quantum field theory and the Jones polynomial},
Comm. Math. Phys.121(1989), no.3, 351--399.

\bibitem{Witten}  E. Witten, {\em On quantum gauge theories in two dimensions}, Comm. Math. Phys., Volume 141, Number 1, 153--209 (1991).


\end{thebibliography}
\end{document}